\documentclass[12pt,twoside]{amsart}
\usepackage{amssymb}
\usepackage{amscd}
\usepackage[abbrev,alphabetic]{amsrefs}
\usepackage{hyperref}
\usepackage{comment}
\usepackage{array,multirow,tabularx,longtable}
\usepackage{tikz}
\usetikzlibrary{cd}
\usepackage{here}
\usepackage{multirow}
\usepackage[margin=1.25in]{geometry}
\usepackage{mathtools}

\usepackage{framed}
\usepackage{fancybox}
\usepackage{ascmac}
\title[Frobenius splitting for Fano threefolds]
{Frobenius splitting for Fano threefolds}
\thanks{The author is supported by JSPS KAKENHI Grant Numbers JP22H01112 and JP23K03028.}

\author{Hiromu Tanaka} 
\address{Department of Mathematics, 
Graduate School of Science, 
Kyoto University, 
Kyoto 606-8502, JAPAN} 
\email{tanaka.hiromu.7z@kyoto-u.ac.jp}
\DeclareMathOperator{\diag}{diag}

\DeclareMathOperator{\Mat}{Mat}
\DeclareMathOperator{\SO}{SO}
\DeclareMathOperator{\Lie}{Lie}
\DeclareMathOperator{\Spin}{Spin}

\DeclareMathOperator{\Stab}{Stab}

\DeclareMathOperator{\Sym}{Sym}
\DeclareMathOperator{\LG}{LG}
\DeclareMathOperator{\Sp}{Sp}

\newcommand{\Alt}[0]{{\operatorname{Alt}}}

\newcommand{\wt}{\widetilde}

\newcommand{\rank}[0]{\operatorname{rank}}

\newcommand{\codim}[0]{{\operatorname{codim}}}

\newcommand{\red}[0]{{\operatorname{red}}}

\newcommand{\Ker}[0]{\operatorname{Ker}}

\newcommand{\Spec}[0]{\operatorname{Spec}}

\newcommand{\Hom}[0]{{\operatorname{Hom}}}

\newcommand{\Pic}[0]{\operatorname{Pic}}

\newcommand{\Gr}[0]{{\operatorname{Gr}}}
\newcommand{\OGr}[0]{{\operatorname{OGr}}}

\newcommand{\GL}[0]{{\operatorname{GL}}}
\newcommand{\SL}[0]{{\operatorname{SL}}}

\newtheorem{thm}{Theorem}[section]
\newtheorem{lem}[thm]{Lemma}
\newtheorem*{lem*}{Lemma}

\newtheorem{prop}[thm]{Proposition}

\newtheorem*{claim*}{Claim}         
\newtheorem{step}{Step}

\theoremstyle{definition}

\newtheorem{dfn}[thm]{Definition}

\newtheorem{rem}[thm]{Remark}

\newtheorem{nota}[thm]{Notation}         
\makeatletter
  
  \@addtoreset{equation}{thm}
  \makeatother

\newcommand{\cA}{\mathcal{A}}
\newcommand{\cB}{\mathcal{B}}

\newcommand{\cE}{\mathcal{E}}
\newcommand{\cF}{\mathcal{F}}

\newcommand{\cI}{\mathcal{I}}

\newcommand{\cK}{\mathcal{K}}
\newcommand{\cL}{\mathcal{L}}
\newcommand{\cM}{\mathcal{M}}
\newcommand{\cN}{\mathcal{N}}
\newcommand{\cO}{\mathcal{O}}

\newcommand{\cQ}{\mathcal{Q}}

\newcommand{\cU}{\mathcal{U}}
\newcommand{\cV}{\mathcal{V}}

\newcommand{\MO}{\mathcal{O}}

\newcommand{\R}{\mathbb{R}}

\newcommand{\Z}{\mathbb{Z}}

\renewcommand{\P}{\mathbb{P}}

\newcommand{\p}{\mathfrak{p}}

\newcommand{\g}{\mathfrak{g}}

\newcommand{\la}{\langle}
\newcommand{\ra}{\rangle}

\DeclareMathOperator{\Fl}{Fl}

\usepackage{listings}
\begin{document}

\maketitle

\begin{abstract}
We prove that every smooth Fano threefold of characteristic $p>5$ is quasi-$F$-split. 
Moreover, if $p>11$, then it is $F$-split. 
\end{abstract}

\tableofcontents

\section{Introduction}

$F$-splitting was introduced by Mehta and Ramanathan \cite{MR85}. 
It has strong consequences for the geometry and cohomology of algebraic varieties in positive characteristic. 
For instance, smooth projective $F$-split varieties satisfy 
Akizuki--Nakano vanishing \cite{Pet}. 
On the other hand, $F$-split varieties form a rather restricted class. 
For example, a smooth projective $F$-split variety necessarily has non-positive Kodaira dimension.
It is therefore natural to ask under what conditions Fano (type) varieties are $F$-split.

In this direction, it is known that toric varieties and Schubert varieties are $F$-split. 
Moreover, every smooth del Pezzo surface in characteristic $p>5$ is $F$-split, and this bound is optimal \cite[Example 5.5]{Har98}. 
The purpose of this paper is to establish a three-dimensional analogue of this result.
More precisely, our main theorem is as follows.

\begin{thm}[Theorem \ref{t QFS main}, Theorem \ref{t Fsplit main}]\label{intro main}
Let $k$ be an algebraically closed field of characteristic $p>0$ and 
let $X$ be a Fano threefold over $k$, i.e., 
$X$ is a $3$-dimensional smooth projective integral scheme 
over $k$ such that $-K_X$ is ample. 
Then the following hold: 
\begin{enumerate}
\item If $p>11$, then $X$ is globally $F$-regular. 
\item If $p>5$, then $X$ is globally quasi-$F$-regular. 
\end{enumerate}
\end{thm}

\begin{rem}
Given a smooth Fano variety $X$ over $k$, 
$X$ is $F$-split (resp.\ quasi-$F$-split) if and only if $X$ is globally $F$-regular (resp.\ globally quasi-$F$-regular). 
\end{rem}

\begin{rem}
The bounds in Theorem \ref{intro main} are optimal. 
Indeed,  the weighted hypersurface
\[
X := \{ x_0^6 + x_1^6 + x_2^6 + x_3^6 +y^2=0 \} \subset \P(1, 1, 1, 1, 3)
\]
is a Fano threefold when $p>3$, and 
the following hold:  
\begin{enumerate}
\item If $p=11$, then $X$ is not $F$-split 
\cite[Example 8.4]{KTLift2}. 
\item If $p=5$, then $X$ is not quasi-$F$-split 
\cite[Example 8.7]{KTLift2}. 
\end{enumerate}
\end{rem}

We say that $X$ is a {\em prime Fano threefold} over $k$ if $X$ is a Fano threefold over $k$ %which is a closed subscheme of $\P^{g+1}_k$ 
such that $\Pic X$ is generated by $\omega_X$. The positive integer $g$ defined by $(-K_X)^3 = 2g-2$ is called the {\em genus} of $X$. 
Theorem \ref{intro main} was established in \cite{KTLift2} 
unless $X$ is a prime Fano threefold such that $|-K_X|$ is very ample. 
Thus it remains to treat prime Fano threefolds such that $|-K_X|$ is very ample, which is the new contribution of this paper. 
In fact, we obtain the following sharper bounds depending on the genus.

\begin{thm}[Theorem \ref{t QFS main}, Theorem \ref{t prime Fsplit main}]\label{intro main prime}
Let $k$ be an algebraically closed field of characteristic $p>0$ and 
let $X$ be a prime Fano threefold of genus $g$ over $k$. 
Then  the following table gives sufficient conditions for $X$ to be
globally $F$-regular or globally quasi-$F$-regular:
{\small
\[
\begin{array}{c|cccccccccc}
g
& 2 & 3 & 4 & 5 & 6 & 7 & 8 & 9 & 10 & 12\\
\hline
\text{GFR}
& p>11 & p>7 & p>5 & p>3 & p>3
& p>2 & p>2 & p>2 & p>3 & p>2\\
\text{GQFR}
& p>5 & p>3 & p>2 & p>0 & p>0
& p>0 & p>0 & p>0 & p>0 & p>0
\end{array}
\]
}
Here GFR (resp.\ GQFR) stands for globally $F$-regular
(resp.\ globally quasi-$F$-regular).
\end{thm}

\subsection{Overview of the proof of Theorem \ref{intro main  prime}}

Let $X$ be a prime Fano threefold of genus $g$ such that $|-K_X|$ is very ample. 
In what follows, we focus on the case when $g\geq6$, 
as otherwise the problem is simpler. 
By the Cartier operator criterion (cf.\ Proposition \ref{p GFS criterion}), it is enough to prove the following: 
\begin{enumerate}
\item[(I)] $H^2(X, \Omega^1_X(m))=0$ if $m\geq 2$. 
\item[(II)] $H^1(X, \Omega^2_X(p))=0$. 
\end{enumerate}

(I) 
Fix $m\geq 2$. 
Let $X \subset \P^{g+1}$ be the closed embedding 
induced by $|-K_X|$. 
By the conormal exact sequence
\[
0
\to 
\cN_{X/\mathbb P^{g+1}}^\vee(m)
\to 
\Omega_{\mathbb P^{g+1}}^1(m)|_X
\to 
\Omega_X^1(m)
\to 
0, 
\]
it suffices to show (i) and (ii) below.  
\begin{enumerate}
\item[(i)] $H^2(X, \Omega_{\mathbb P^{g+1}}^1(m)|_X)=0$. 
\item[(ii)] $H^3(X, \cN_{X/\mathbb P^{g+1}}^\vee(m))=0$. 
\end{enumerate}
The Euler sequence on $\mathbb P^{g+1}$ induces 
\[
0
\to 
\Omega_{\mathbb P^{g+1}}^1(m)|_X
\to 
\cO_X(m-1)^{\oplus(g+2)}
\to 
\cO_X(m)
\to 
0, 
\]
which implies  (i). Since $X \subset \P^{g+1}$ is an intersection of quadrics, 
there exists a surjection
$\cO_{\mathbb P^{g+1}}^{\oplus r}
\twoheadrightarrow
\cI_{X/\P^{g+1}}(2)$ for some $r>0$. Hence we get a surjection 
\[
\cO_X(m-2)^{\oplus r}
\twoheadrightarrow
\cI_X/\cI_X^2(m)
=
\cN_{X/\mathbb P^{g+1}}^\vee(m). 
\]
Then we get  $0= H^3(X, \MO_X(m-2))^{\oplus r} 
\twoheadrightarrow
H^3(X, \cN_{X/\mathbb P^{g+1}}^\vee(m))$. 
Thus (ii) holds.

\medskip

(II, $g=7$) 
As the proofs of $g \in \{6, 7, 9\}$ are quite similar, 
we overview the proof for the case when $g=7$. 
In this case, $X$ is a linear section of $\Sigma := \OGr^+(5, 10) \subset \P^{15}$ of codimension $7$. 
Applying $\wedge^2$ to the conormal exact sequence $0 \to \MO_X(-1)^{\oplus 7} \to \Omega_\Sigma^1|_X \to \Omega^1_X \to 0$, 
we get the following exact sequences for some vector bundle $\cF$: 
\[
0 \to \cF \to \Omega_\Sigma^2|_X \to \Omega^2_X \to 0, \qquad 
0 \to \MO_X(-2)^{\oplus 21} 
\to \cF \to \Omega^1_X(-1)^{\oplus 7} \to 0. 
\]
Moreover, we have the Koszul exact sequence
\[
0
\to
\Omega^2_\Sigma(-7)
\to
\Omega^2_\Sigma(-6)^{\oplus 7}
\to
\Omega^2_\Sigma(-5)^{\oplus 21}
\to
\Omega^2_\Sigma(-4)^{\oplus 35}
\]
\[
\to
\Omega^2_\Sigma(-3)^{\oplus 35}
\to
\Omega^2_\Sigma(-2)^{\oplus 21}
\to
\Omega^2_\Sigma(-1)^{\oplus 7}
\to
\Omega^2_\Sigma
\to
\Omega^2_\Sigma|_X
\to0. 
\]
Then it is easy to see that the problem is reduced to showing 
\[
H^{\ell+1}(\Sigma, \Omega^2_\Sigma(p-\ell))=0
\text{ \qquad if \qquad}
0 \leq \ell \leq 7.
\]

For the full flag variety $F := \Spin_{10}/B$ of type $D_5$ and the projection $\pi: F= \Spin_{10}/B \to \Spin_{10}/P = \Sigma$, 
we can check that 
\[
H^i(\Sigma, \Omega^2_\Sigma(p-\ell)) \simeq H^i(F, \cL_{\omega_2+(p-\ell-3)\omega_4+\omega_5})
\]
for the standard choice of fundamental weights $\omega_1, ..., \omega_5$. 
If $\cL_{\omega_2+(p-\ell-3)\omega_4+\omega_5} \otimes \MO_F(-K_F)$ is ample, then  the required equality 
$H^i(F, \cL_{\omega_2+(p-\ell-3)\omega_4+\omega_5})=0$ 
is ensured by the Kodaira vanishing for flag varieties. 
However, we also need some cases in which this condition does not hold, e.g., $(p, \ell) = (3, 7)$. 
%\in \{-1, -2, -3, -4\}$. 

In what follows, we treat the case when 
$(p, \ell) = (3, 7)$, i.e., we overview the proof of 
$H^i(F, \cL_{\lambda})=0$ for 
\[
\lambda := \omega_2-7\omega_4+\omega_5 = (-2,-2,-3,-3,4). 
\]
Suppose $H^i(F, \cL_\lambda) \neq 0$ for some $i$. 
Then 
$H^i(F, \cL_\lambda)$ has some composition factor 
$L(\mu)$ of highest weight $\mu$. 
By the strong linkage principle (Theorem \ref{t SLP}), 
we obtain $\mu \leq \lambda^+$, i.e., 
\[
\lambda^+-\mu
=
\sum_{j=1}^5c_j\alpha_j
\]
for some $c_1,\ldots,c_5\in\mathbb Z_{\geq0}$. 
By using $\rho=\omega_1+\cdots+\omega_5=(4, 3, 2, 1, 0)$,  
$\lambda^+$ can be computed as follows (cf.\ Subsection \ref{ss SLP}): 
\[
\begin{array}{c|c|c|c}
\lambda & \lambda+\rho & \lambda^++\rho & \lambda^+\\
\hline
(-2,-2,-3,-3,4)
& (2,1,-1,-2,4)
& (4,2,2,1,1)
& (0,-1,0,0,1)
\end{array}
\]
For the linear function $f: \R^5 \to \R$ defined by  $f(x_1,\ldots,x_5):=x_1+x_2$, 
we have  $f(\lambda^+)=-1$, $f(\alpha_2)=1$, and 
$f(\alpha_1) = f(\alpha_3)=f(\alpha_4)=f(\alpha_5)=0$. 
As $\mu$ is dominant, we have $f(\mu) \geq 0$, 
which leads to the following contradiction: 
\[
-1 \geq -1-f(\mu)
=f(\lambda^+)- f(\mu)
=f(\lambda^+-\mu)
=\sum_{j=1}^5 c_j f(\alpha_j)
=c_2\geq0. 
\]

\medskip

(II, $g=12$) 
By the conormal sequence and Koszul sequence, 
the problem is reduced to proving 
\[
H^{>0}(\Gr(4, 7), 
\cQ^{\otimes k} \otimes \MO_{\Gr(4, 7)}(t-1)) =0 
 \]
for every $k \geq 0$ and $t \geq 0$, 
or a similar vanishing result. 
Given a dominant weight $\lambda$ of $\GL_3$, we set 
$\nabla(\lambda) := H^0(\GL_3/B_{\GL_3}, \cL_{\lambda})$. 
For 
$o:=eP\in\GL_7/P$ and the natural projection 
\[
\pi : E := \Fl(4, 5, 6; 7) =\GL_7/Q \to \Gr(4, 7) = \GL_7/P, 
\]
we have $\cQ_o \simeq \nabla(1, 0, 0)$. 

We say that a $\GL_3$-module $V$ has a
{\em good filtration} if there exists a sequence
$V=:F_0\supset F_1\supset\cdots\supset F_N=0$
of $\GL_3$-submodules such that
$F_i/F_{i+1}\simeq\nabla(\lambda_i)$ for every $i$, 
where each $\lambda_i$ is a dominant weight of $\GL_3$.
Tensor products of $\GL_3$-modules with good
filtrations again have good filtrations
\cite[Corollary 6.3]{Mat00}.
Consequently, the $\GL_3$-module
$\cQ_o^{\otimes k}$  has a good filtration.
As $\cQ_o^{\otimes k}$ is a polynomial
$\GL_3$-module,
any weight $\mu = (\mu_1, \mu_2, \mu_3)$ of $\cQ_o^{\otimes k}$ satisfies 
$\mu_i\geq0$ for every $i$. % \cite[(3.2c)]{Gre07}.
Then this property also holds for the filtration quotients,
and hence their highest weights
$\lambda=(\lambda_1,\lambda_2,\lambda_3)$ satisfy
$\lambda_1\geq \lambda_2\geq\lambda_3\geq0$. 
By standard argument, we have 
the following $P$-module isomorphisms 
for the pullback
$\lambda'\in X(Q)$ of
$\lambda\in X(B_{\GL_3})$: 
\[
\bigl(\pi_*\cL_{\lambda'}\bigr)_o
\simeq
H^0(P/Q,\cL_{\lambda'}|_{P/Q})
\simeq
H^0(\GL_3/B_{\GL_3},\cL_\lambda)
=
\nabla(\lambda). 
\]
Corresponding to the good filtration
$\cQ^{\otimes k}_o=:F_0\supset F_1
\supset\cdots\supset F_N=0$,
we get a sequence of locally free $\MO_{\Gr(4, 7)}$-submodules
$\cQ^{\otimes k}=:\cF_0\supset\cF_1
\supset\cdots\supset\cF_N=0$
whose graded quotients are of the form
$\pi_*\cL_{\lambda'}$. 
By the Kodaira vanishing on $E = \Fl(4, 5, 6; 7)$, 
it is enough to show that each $\cL_{\lambda'}$
is nef, which can be checked by using $\lambda_1\geq \lambda_2\geq\lambda_3\geq0$.

\medskip
\noindent {\bf Acknowledgements.}
The author was supported by JSPS KAKENHI Grant number JP22H01112 and JP23K03028. 
The author thanks Akihiro Kanemitsu for fruitful discussions. 
ChatGPT (OpenAI) was used during the preparation of this article for assistance with mathematical exploration, computations, and language editing.
\section{Preliminaries}

\subsection{Notation}\label{ss-notation}

In this subsection, we summarise notation used in this paper. 

\begin{enumerate}
\item We will freely use the notation and terminology in \cite{Har77} and \cite{KM98}. 
In particular, $D_1 \sim D_2$ means linear equivalence of Weil divisors. 
\item 
Throughout this paper, 
we work over an algebraically closed field $k$ 
of characteristic $p>0$ unless otherwise specified. 
\item For an integral scheme $X$, 
we define the {\em function field} $K(X)$ of $X$ 
as the local ring $\MO_{X, \xi}$ at the generic point $\xi$ of $X$. 
For an integral domain $A$, $K(A)$ denotes the function field of $\Spec\,A$. 

\item 
For a scheme $X$, its {\em reduced structure} $X_{\red}$ 
is the reduced closed subscheme of $X$ such that the induced closed immersion 
$X_{\red} \to X$ is surjective. 
\item 
Our notation will not distinguish between invertible sheaves and 
Cartier divisors. For example, we will write $L+D$ for 
an invertible sheaf $L$ and a Cartier divisor $D$. 
\item We say that $X$ is a {\em variety} (over $k$) if 
$X$ is a separated integral scheme which is of finite type over $k$. 
We say that $X$ is a {\em curve} (resp. a {\em surface}, resp. a {\em threefold})  
if $X$ is a variety over $k$ of dimension one (resp. {\em two}, resp. {\em three}). 
\item We say that $f: X \to Y$ is a {\em contraction} if $f$ is a projective morphism of schemes satisfying $f_*\MO_X = \MO_Y$. 
Here the equality $f_*\MO_X = \MO_Y$ means that the induced ring homomorphism $\MO_Y \to f_*\MO_X$ is an isomorphism. 

\item For a smooth projective variety $V$ satisfying  $\Pic V \simeq \Z$, 
$\MO_V(1)$ denotes the ample invertible sheaf which  generates $\Pic\,V (\simeq \Z)$ and 
we set $\MO_V(\ell) := \MO_V(1)^{\otimes \ell}$. 
\item 
An {\em algebraic group} is a smooth group scheme of finite type over $k$. 
We say that $G$ is a {\em reductive} (resp.\ {\em semisimple}) group if $G$ is a connected algebraic group 
whose unipotent radical $R_u(G)$ 
(resp.\ its radical $R(G)$) is trivial. 
We say that $F$ is a {\em flag variety} 
if $F \simeq G/P$ for some reductive group $G$ and a 
reduced parabolic subgroup $P$ of $G$. 
For some foundational results on algebraic groups and their representations, 
we refer to \cite{Spr98}, \cite{Mil17}, \cite{Jan03}.
\end{enumerate}

\subsection{Fano threefolds}

 We say that $X$ is a {\em Fano threefold}  
if $X$ is a three-dimensional smooth projective variety over $k$ such that $-K_X$ is ample. 
For a Fano threefold $X$, 
the following holds: 
\begin{enumerate}
\item 
The {\em index} $r_X$ of $X$ is the largest positive integer $r$
that divides $-K_X$ in $\Pic\,X$, and it satisfies
$1\leq r_X\leq4$ \cite[Theorem 2.18]{FanoI}.
\item $\Pic X \simeq \Z^{\oplus \rho(X)}$ \cite[Theorem 2.4]{FanoI}. 
\item 
For $m \geq 0$, we have 
\[
h^0(X, -mK_X) = \chi(X, -mK_X) = \frac{1}{12} m(m+1)(2m+1)(-K_X)^3 + 2m + 1
\]
by \cite[Corollary 2.6]{FanoI} and Theorem \ref{t Fano3 ANV} below. 
\end{enumerate}

\begin{thm}\label{t Fano3 ANV}
Let $X$ be a Fano threefold and let $A$ be an ample Cartier divisor. 
Then $H^j(X, \Omega_X^i(A))= 0$ for every $i+j >3$. 
\end{thm}

\begin{proof}
The assertion follows from \cite{KTLift1} and \cite{KTLift2}. 
\end{proof}

We say that $X$ is a {\em prime Fano threefold} 
if $X$ is a Fano threefold such that $\Pic X  = \Z K_X$. 
For a prime Fano threefold $X$, 
the integer $g$ defined by $(-K_X)^3 = 2g -2$ is called the {\em genus} of $X$. 
It is well known that the following hold  \cite[Theorem 1.1]{FanoI}, \cite[Theorem 1.2]{FanoII}: 
\begin{enumerate}
\item[(i)] $2 \leq g \leq 12$ and $g \neq 11$. 
\item[(ii)] 
If $|-K_X|$ is very ample, then $g \geq 3$. 
If $g \geq 4$, then $|-K_X|$ is very ample. 
\end{enumerate}
If $X$ is a prime Fano threefold and $|-K_X|$ is very ample, 
then $X$ has a concrete description of Mukai type 
\cite[Theorem 1.2]{KTprime} ($g \geq 6$), \cite[Proposition 2.8]{FanoII} ($g \leq 5$).

\begin{prop}\label{p GFS criterion}
Let $X$ be a Fano threefold. 
Then the following hold: 
\begin{enumerate}
\item If $H^2(X, \Omega_X^1(-p^rK_X))=0$ for every $r>0$, 
then $X$ is quasi-$F$-split. 
\item If $H^2(X, \Omega_X^1(-pK_X)) =H^1(X, \Omega_X^2(-pK_X))=0$, then $X$ is $F$-split. 
\end{enumerate}
\end{prop}

\begin{proof}
It holds that 
$H^1(X, \Omega_X^1(K_X))=0$ and 
$H^0(X, \Omega_X^2(mK_X))=0$
for every $m>0$ \cite[Theorem B]{KTLift1}. 
Then the assertion (1) (resp. (2)) 
follows from \cite[Proposition 2.20]{KTLift2} (resp. 
\cite[Proposition 2.15]{KTLift2}).  
\end{proof}

\subsection{Cohomologies of line bundles on flag varieties}

\begin{dfn}
Let $G$ be a connected reductive algebraic group, let $T\subset B\subset G$ be a maximal torus and a Borel subgroup. 
Fix a system $\Phi^+$ of positive roots such that 
$B$ is the corresponding negative Borel subgroup. 
Set
$\rho
:=
\frac{1}{2}
\sum_{\alpha\in\Phi^+}\alpha$. 
For the Weyl group $W:=N_G(T)/T$, 
the {\em dot action} of $W$ on the weight lattice $X(T)$ is defined by 
\[
w\cdot\lambda
:=
w(\lambda+\rho)-\rho
\]
for $w\in W$ and $\lambda\in X(T)$.
\end{dfn}

\noindent
For a simple root $\alpha$ and the corresponding simple reflection $s_\alpha$, we have
\[
s_\alpha\cdot\lambda
=
\lambda
-
(
\left\langle
\lambda,\alpha^\vee
\right\rangle
+1
)
\alpha.
\]

\begin{prop}\label{p positive characteristic reflection formula}
Let $G$ be a reductive algebraic group.
Fix a maximal torus $T\subset B\subset G$, where $B$ is a Borel subgroup, and set $F:=G/B$. 
For a weight $\lambda\in X(T)$, let $\cL_\lambda$ be the corresponding line bundle on $F$.
Let $\alpha$ be a simple root and let $s_\alpha$ be the corresponding simple reflection.
Then the following hold: 
\begin{enumerate}
\item
If $-p
\leq
\left\langle
\lambda,\alpha^\vee
\right\rangle
\leq
-2$, then $H^i(
F,
\cL_\lambda
)
\simeq
H^{i-1}(
F,
\cL_{s_\alpha\cdot\lambda}
)$ for every integer $i$. 
\item
If $\left\langle
\lambda,\alpha^\vee
\right\rangle
=
-1,$ then $H^i(
F,
\cL_\lambda
)
=
0$ for every integer $i$.
\end{enumerate}
\end{prop}

\begin{proof}
The assertion (1) and (2) follow from 
(d) and (a) of \cite[Proposition II.5.4]{Jan03}, respectively. 
\end{proof}

\begin{prop}\label{p flag Kodaira}
Let $X$ be a flag variety and let 
$f : X \to Y$ be a morphism to a projective scheme $Y$. 
Let $L$ be a Cartier divisor 
such that $L-K_X$ is $f$-ample or $L$ is $f$-nef. 
Then $R^if_*\MO_X(L) = 0$ for every $i>0$. 
\end{prop}

\begin{proof}
Since $-K_X$ is ample, we may assume that $L-K_X$ is $f$-ample. 
In this case, the assertion follows from \cite[Theorem 1 in page 38]{MR85}. 
\end{proof}

\begin{prop}\label{p almost ample general}
Let $G$ be a reductive algebraic group and let
$P\subset Q\subset G$ be reduced parabolic subgroups.
Let $\pi:G/P\to  G/Q$ be the natural projection. 
Let $M$ be an ample Cartier divisor on $G/Q$ and set
$L:=\pi^*M.$ 
Then
\[
H^i(
G/P,
K_{G/P}+L
)=0 \qquad \text{and} \qquad H^d(
G/P,
K_{G/P}+L
)
\simeq
H^0(
G/Q,
K_{G/Q}+M
)
\]
for every $i\neq d$, where $d:=\dim(Q/P)$.
\end{prop}

\begin{proof}
The assertion follows from 
Proposition \ref{p flag Kodaira} and 
\begin{align*}
H^i(
G/P,
K_{G/P}+L
) 
&\simeq 
H^{\dim G/P -i}(G/P, -\pi^*M)^\vee\\
&\simeq H^{\dim G/P -i}(G/Q, -M)^\vee\\
&\simeq H^{\dim G/Q -(\dim G/P -i)} (G/Q, K_{G/Q} +M)\\
&= H^{i-d}(G/Q, K_{G/Q} +M). 
\end{align*}

\end{proof}

\begin{prop}\label{p almost ample}
Let $F:=G/B$, where 
$G$ is a semisimple group and $B$ is a Borel subgroup. 
Let $\omega_0, ..., \omega_n$ be the fundamental weights. 
Take positive integers $a_1, ..., a_n$. 
For the Cartier divisor $L := L_{a_1 \omega_1+ \cdots a_n\omega_n}$ 
corresponding to the weight $a_1 \omega_1+ \cdots a_n\omega_n$, 
\[
H^i(F, K_F + L)=0
\]
for every $i \neq 1$. 
\end{prop}

\begin{proof}
Let $\pi : F \to E$ be the contraction induced by $L$. 
Then $\pi$ is a $\P^1$-bundle and $L \sim \pi^*M$ for some ample Cartier divisor $M$ on $E$. 
Then the assertion follows from Proposition \ref{p almost ample general}. 
\end{proof}

\begin{rem}\label{r non-diag vanishing}
\begin{enumerate}
\item 
Given a flag variety $X$, we have $H^j(X, \Omega_X^i)=0$ 
for $i \neq j$ \cite[Theorem 2 in page 7]{Lau97}. 
\item 
Given integers $i, j, n, a$ satisfying $0<j<n$, we have 
\[
H^j(\P^n, \Omega_{\P^n}^i(a))=0
\]
if $i \neq j$ or $a \neq 0$. Indeed, this follows from 
(1) and the Bott vanishing theorem. 
\end{enumerate}
\end{rem}

\subsection{Strong linkage principle}\label{ss SLP}

Let $G$ be a reductive group. 
Take a maximal torus $T$ and a Borel subgroup $B$ satisfying $G \supset B \supset T$. 
For $\lambda \in X(T)$, 
we define $\lambda^+$ as the unique element $\lambda^+ \in X(T)$ such that 
\begin{enumerate}
\item $\lambda^+ + \rho$ is dominant, and 
\item $\lambda^+$ belongs to the dot action orbit of $\lambda$ by the Weyl group $W$, i.e., 
\[
\lambda^+ + \rho = w(\lambda + \rho)
\]
for some $w \in W$. 
\end{enumerate}
For example, if $\lambda$ is dominant, then $\lambda = \lambda^+$.

\begin{thm}\label{t SLP}
Let $G$ be a reductive group. 
Take a maximal torus $T$ and a Borel subgroup $B$ satisfying $G \supset B \supset T$. 
For a weight $\lambda \in X(T)$ and a dominant weight $\mu \in X(T)$, 
assume that the simple $G$-module $L(\mu)$ of highest weight $\mu$ is a composition factor of $H^i(G/B, \cL_\lambda)$ for some $i$. 
Then the following hold: 
\begin{enumerate}
\item $\mu \leq \lambda^+$. 
\item $\lambda + \rho = w(\mu + \rho) + p \gamma$ 
for some $w \in W$ and 
$\gamma \in \Z \Phi$, 
where $W$ and $\Phi$ denote the Weyl group and the set of roots 
of $(G, T)$.
\end{enumerate}
\end{thm} 

\begin{proof}
See 
\cite[Theorem 5.1]{And23} or \cite[Proposition II.6.13]{Jan03}. 
\end{proof}

\begin{rem}
With the above notation, 
the $W$-dot-action-orbits of $\lambda$ and $\lambda^+$ are the same. 
The condition (2) of Theorem \ref{t SLP} means that 
the $W_p$-dot-action-orbits of $\lambda$ and $\mu$ are the same for the affine Weyl group $W_p:=W\ltimes p\mathbb Z\Phi$. 
In particular,  $\lambda, \lambda^+$, and $\mu$ belong to a single $W_p$-dot-action-orbit. 
\end{rem}

\section{Case $g \leq 5$ and quasi-F-splitting}

\begin{prop}\label{p quasi-F-split quadrics}
Let $X$ be a prime Fano threefold of genus $g \geq 5$. 
%over an algebraically closed field $k$of characteristic $p>0$.
Then the following hold: 
\begin{enumerate}
\item $H^2(
X,
\Omega_X^1(m)
)
=0$ for every $m \geq 2$. 
\item $X$ is quasi-$F$-split. 
\item If $H^1(X, \Omega_X^2(p))=0$, then $X$ is $F$-split. 
\end{enumerate}
%In particular, $X$ is quasi-$F$-split.
\end{prop}

\begin{proof}
As (1) implies (2) and (3) (Proposition \ref{p GFS criterion}), let us show (1). 
Let $X \subset \P^{g+1}$ be the closed embedding 
induced by $|-K_X|$. 
By the conormal exact sequence
\[
0
\to 
\cN_{X/\mathbb P^{g+1}}^\vee(m)
\to 
\Omega_{\mathbb P^{g+1}}^1(m)|_X
\to 
\Omega_X^1(m)
\to 
0, 
\]
it suffices to show (i) and (ii) below.  
\begin{enumerate}
\item[(i)] $H^2(X, \Omega_{\mathbb P^{g+1}}^1(m)|_X)=0$ for every $m \in \Z$. 
\item[(ii)] $H^3(X, \cN_{X/\mathbb P^{g+1}}^\vee(m))=0$ for every $m \geq 2$. 
\end{enumerate}

Let us show (i). 
Restricting the Euler sequence on $\mathbb P^{g+1}$ to $X$, we obtain
\[
0
\to 
\Omega_{\mathbb P^{g+1}}^1(m)|_X
\to 
\cO_X(m-1)^{\oplus(g+2)}
\to 
\cO_X(m)
\to 
0.
\]
As $H^1(X, \MO_X(n)) = H^2(X, \MO_X(n))=0$ for every $n \in \Z$, 
we get $H^2(X, \Omega_{\mathbb P^{g+1}}^1(m)|_X)=0$ for every $m\in \Z$. 
Thus (i) holds.

Let us show  (ii). 
Since $X \subset \P^{g+1}$ is an intersection of quadrics, 
there exists a surjection
$\cO_{\mathbb P^{g+1}}^{\oplus r}
\twoheadrightarrow
\cI_{X/\P^{g+1}}(2)$ for some integer $r>0$. 
Taking the restriction to $X$, we obtain a surjection
\[
\cO_X(-2)^{\oplus r}
\twoheadrightarrow
\cI_X/\cI_X^2
=
\cN_{X/\mathbb P^{g+1}}^\vee.
\]
As the functor $H^3(X, -)$ preserves surjectivity, 
it is enough to prove $H^3(X, \MO_X(m-2))=0$ for every $m\geq 2$, 
which follows from Serre duality. 
Thus (ii) holds. 
\end{proof}

\begin{prop}\label{p genus 3 F-split}
Let $X \subset \P^4$ be a smooth quartic hypersurface.  
%For every integer $m\geq4$, we have
Then the following hold: 
\begin{enumerate}
\item 
$H^2(
X,
\Omega_X^1(m)
)
=0$ for every $m \geq 4$. 
\item 
$H^1(
X,
\Omega_X^2(m)
)
=0$ for every $m \geq 8$. 
\item 
$X$ is quasi-$F$-split if $p>3$. 
\item 
$X$ is $F$-split if $p>7$. 
\end{enumerate}
\end{prop}

\begin{proof}
Since (1) and (2) imply (3) and (4) 
(Proposition \ref{p GFS criterion}), 
it suffices to prove (1) and (2).

Let us show (1). 
Fix an integer $m \geq 4$. 
By the conormal exact sequence %of $X\subset\P^4$ is
$0
\to 
\cO_X(-4)
\to 
\Omega_{\P^4}^1|_X
\to 
\Omega_X^1
\to 
0$, 
it is enough to show $H^2(X, 
\Omega_{\P^4}^1(m)|_X)=0$ and $H^3(X, \MO_X(m-4))=0$. 
The latter one follows from $m \geq 4$ and Serre duality. 
It suffices to prove 
$H^2(X, \Omega_{\P^4}^1(m)|_X)=0$, which 
follows from Remark \ref{r non-diag vanishing} and the restriction exact sequence
\[
0
\to 
\Omega_{\P^4}^1(m-4)
\to 
\Omega_{\P^4}^1(m)
\to 
\Omega_{\P^4}^1(m)|_X
\to 
0. 
\]
Thus (1) holds.

Let us show (2). 
Fix an integer $m\geq 8$. 
Applying $\wedge^2$ to the conormal exact sequence 
$0
\to 
\cO_X(-4)
\to 
\Omega_{\P^4}^1|_X
\to 
\Omega_X^1
\to 
0$,
we obtain
\[
0
\to 
\Omega_X^1(-4)
\to 
\Omega_{\P^4}^2|_X
\to 
\Omega_X^2
\to 
0.
\]
By (1), 
%$H^2(X,\Omega_X^1(m-4)))=0$ (Lemma \ref{l genus 3 Omega1}), 
it suffices to show $H^1(
X,
\Omega_{\P^4}^2(m)|_X
)
=0.$ 
This follows from 
Remark \ref{r non-diag vanishing} and the restriction exact sequence
\[
0
\to 
\Omega_{\P^4}^2(m-4)
\to 
\Omega_{\P^4}^2(m)
\to 
\Omega_{\P^4}^2(m)|_X
\to 
0. 
\]
%and the Bott vanishing theorem $H^{>0}(\P^4, \Omega^2_{\P^4}(n))=0$ ($n>0$). 
Thus (2) holds. 
\end{proof}

\begin{prop}\label{p genus 4 F-split}
Let $X$ be a prime Fano threefold of genus $4$. 
Then the following hold: 
\begin{enumerate}
\item 
$H^2(
X,
\Omega_X^1(m)
)=0$ for every integer $m\geq3$. 
\item 
$H^1(
X,
\Omega_X^2(m)
)
=0$ for every integer $m\geq6$. 
\item If $p>2$, then $X$ is quasi-$F$-split. 
\item 
If $p>5$, then $X$ is $F$-split. 
\end{enumerate}
\end{prop}

\begin{proof}
Since (1) and (2) imply (3) and (4) 
(Proposition \ref{p GFS criterion}), 
it suffices to prove (1) and (2).

Let us show (1). 
Fix an integer $m \geq 3$. 
The conormal exact sequence of $X\subset\P^5$ is given by 
\[
0
\to 
\cO_X(-2)\oplus\cO_X(-3)
\to 
\Omega_{\P^5}^1|_X
\to 
\Omega_X^1
\to 
0.
\]
As $H^3(X, \MO_X(m-2))=H^3(X, \MO_X(m-3))=0$, 
it suffices to show $H^2(X, \Omega_{\P^5}^1(m)|_X)=0$. 
By the Koszul exact sequence 
\[
0
\to 
\cO_{\P^5}(-5)
\to 
\cO_{\P^5}(-2)\oplus\cO_{\P^5}(-3)
\to 
\cO_{\P^5}
\to 
\cO_X
\to 
0, 
\]
it suffices to show 
\[
H^2(\P^5, \Omega^1_{\P^5}(m)) =
H^3(\P^5, \Omega^1_{\P^5}(m-2))=
H^3(\P^5, \Omega^1_{\P^5}(m-3))=
H^4(\P^5, \Omega^1_{\P^5}(m-5))=0, 
\]
which follows from Remark \ref{r non-diag vanishing}. 
Thus (1) holds.

Let us show (2). 
Fix an integer $m \geq 6$. 
Set $\cE:=
\cN_{X/\P^5}^{\vee}
\simeq
\cO_X(-2)\oplus\cO_X(-3).$ 
Applying $\wedge^2$ to the conormal exact sequence, we obtain
the following exact sequences for some vector bundle $\cF$:
\[
0
\to 
\cF
\to 
\Omega_{\P^5}^2|_X
\to 
\Omega_X^2
\to 
0,\qquad 
0
\to 
\wedge^2\cE
\to 
\cF
\to 
\cE\otimes\Omega_X^1
\to 
0.
\]
As $\wedge^2\cE
\simeq
\cO_X(-5)$, 
we get $H^2(X, \wedge^2\cE(m)) =H^2(X, \cE \otimes \Omega_X^1(m))=0$ by (1).

It suffices to  show
$H^1(
X,
\Omega_{\P^5}^2(m)|_X
)
=0.$ 
The Koszul exact sequence is given by 
\[
0
\to 
\Omega_{\P^5}^2(m-5)
\to 
\Omega_{\P^5}^2(m-2)
\oplus
\Omega_{\P^5}^2(m-3)
\to 
\Omega_{\P^5}^2(m)
\to 
\Omega_{\P^5}^2(m)|_X
\to 
0.
\]
Since $m\geq6$, all the higher cohomologies except for $\Omega_{\P^5}^2(m)|_X$ are zero by 
Remark \ref{r non-diag vanishing}. 
Therefore, $H^{>0}(X, \Omega_{\P^5}^2(m)|_X)=0.$ 
Thus (2) holds. 
\end{proof}

\begin{prop}\label{p genus 5 Omega2}
Let $X$ be a prime Fano threefold of genus $5$. 
Then the following hold: 
\begin{enumerate}
\item $H^1(
X,
\Omega_X^2(m)
)
=0$ for every  $m\geq 5$. 
\item 
If $p>3$, then $X$ is $F$-split. 
\end{enumerate}
\end{prop}

\begin{proof}
Since (1) and Proposition \ref{p quasi-F-split quadrics} imply (2) (Proposition \ref{p GFS criterion}), let us show (1). 
Fix $m \geq 5$. 
Set $
\cE:=
\cN_{X/\P^6}^{\vee}
\simeq
\cO_X(-2)^{\oplus3}.$ 
Applying $\wedge^2$ to the conormal exact sequence 
$0
\to 
\cE
\to 
\Omega_{\P^6}^1|_X
\to 
\Omega_X^1
\to 
0$, 
we get  the following exact sequences for some vector bundle $\cF$:
\[
0
\to 
\cF
\to 
\Omega_{\P^6}^2|_X
\to 
\Omega_X^2
\to 
0, \qquad 
0
\to 
\wedge^2\cE
\to 
\cF
\to 
\cE\otimes\Omega_X^1
\to 
0.
\]
We have
$\wedge^2\cE
\simeq
\cO_X(-4)^{\oplus3}.$ 
Hence $H^2(X, \wedge^2\cE(m)) = H^2(X, 
\cE\otimes\Omega_X^1(m))=0$ (Proposition \ref{p quasi-F-split quadrics}).

It suffices to show
$H^1(
X,
\Omega_{\P^6}^2(m)|_X
)
=0.$
The Koszul exact sequence of $X\subset\P^6$ is given by 
\[
0
\to 
\cO_{\P^6}(-6)
\to 
\cO_{\P^6}(-4)^{\oplus3}
\to 
\cO_{\P^6}(-2)^{\oplus3}
\to 
\cO_{\P^6}
\to 
\cO_X
\to 
0.
\]
Hence it is enough to prove 
\[
H^{1 +\ell}(\P^6, \Omega^2_{\P^6}(m-2\ell)) =0
\]
for every $\ell \in \{ 0, 1, 2, 3\}$. 
By $0< 1+\ell < 6$, 
this follows from Remark \ref{r non-diag vanishing} when $m -2\ell \neq 0$. 
If $m -2\ell =0$, then it suffices to show $1+ \ell \neq 2$ (Remark \ref{r non-diag vanishing}), which holds by 
$1+\ell = 1+m/2 \geq 1+5/2 >2$. 
Thus (1) holds. 
\end{proof}

\begin{thm}\label{t QFS main}
Let $X$ be a Fano threefold. 
Then the following hold: 
\begin{enumerate}
\item 
If $X$ is not quasi-$F$-split, 
then $\Pic X = \Z K_X$, 
$g \in \{2, 3, 4\}$, and 
the following hold 
%he pair $(g, p)$ satisfies one of the following 
%$g \in \{2, 3, 4\}$ 
for $g := \frac{(-K_X)^3}{2} +1$: 
\begin{enumerate}
\item If $g=2$, then $p \in \{2, 3, 5\}$. 
\item If $g=3$, then $p \in \{2, 3\}$. 
\item If $g=4$, then $p=2$. 
\end{enumerate}
\item 
If $p>5$, then $X$ is quasi-$F$-split. 
\end{enumerate}
\end{thm}

\begin{proof}
Let us show (1). 
Assume that $X$ is not quasi-$F$-split. 
Then we get $\Pic X = \Z K_X$ by \cite[Theorem E]{KTLift2}. 
It follows from Proposition \ref{p quasi-F-split quadrics} that $g \leq 4$. 
If $|-K_X|$ is not very ample (resp.\ is very ample), 
then (a)-(c) hold by \cite[Proposition 6.16]{KTLift2} 
(resp.\ Proposition \ref{p genus 3 F-split} and Proposition \ref{p genus 4 F-split}). 
Thus (1) holds. 
The assertion (2) immediately follows from (1). 
% Let us show (2). 
% By (1), we may assume that 
% $X$ is a prime Fano threefold of genus $g \in  \{2, 3, 4\}$. 
% If $|-K_X|$ is very ample, then 
% we have $g \in \{3, 4\}$, and hence the assertion follows from Proposition \ref{p genus 3 F-split} and Proposition \ref{p genus 4 F-split}. 
% If $|-K_X|$ is not very ample, then 
% the assertion holds by \cite[Proposition 6.16]{KTLift2}. 
\end{proof}

\section{Genus $6$}

Let $X$ be a prime Fano threefold of genus $6$. 
The purpose of this section is to prove that $X$ is $F$-split 
when $p>3$ (Theorem \ref{t genus 6 F-split}). 
Recall that one of the following holds \cite[Theorem 5.5]{KTLift1}:
\begin{enumerate}
\item
%There exists a smooth complete intersection
We have a complete intersection 
\[
X=
\Sigma\cap H_1\cap H_2\cap Q,
\]
where $\Sigma:=\Gr(2,5)$, $H_1$ and $H_2$ are hyperplanes, and $Q$ is a quadric hypersurface.
\item
There exists a separable double cover
$f\colon X\to V,$ 
where $V$ is the smooth quintic del Pezzo threefold.
\end{enumerate}
These cases will be treated separately in 
Subsection \ref{ss g=6 CI} and
Subsection \ref{ss g=6 2-to-1}. 
We start by establishing vanishing results 
Lemma \ref{l genus 6 Grassmannian Omega1} and Lemma \ref{l genus 6 Grassmannian Omega2} for $\Gr(2, 5)$ 
used in both cases.

\begin{nota}\label{n weight lattice on GL5}
Let $T  \subset \SL_5$ be the diagonal maximal torus and 
let $B \subset \SL_5$ be the lower triangular Borel subgroup. 
\begin{enumerate}
\item 
For $1\leq i\leq 5$, let $\epsilon_i\in X(T)$ be the character defined by $\epsilon_i(\diag(t_1,\ldots,t_5)):=
t_i.$ 
For $\delta := (1, 1, 1, 1, 1) \in \Z^5$, 
%we define $[a_1,\ldots,a_5]:= (a_1,..., a_5) + \Z \delta \in \Z^5 / \Z \delta$, 
we use the following identification: 
\[
X(T)
%=(\bigoplus_{i=1}^5\Z e_i)/ \Z \delta 
\simeq 
\Z^5 / \Z \delta, \quad 
a_1\epsilon_1+\cdots+a_5\epsilon_5 \leftrightarrow 
[a_1,\ldots,a_5] := (a_1,..., a_5) + \Z \delta. 
\]
For $\alpha_{ij} := \epsilon_i -\epsilon_j$ 
and 
$\alpha_i
:=
\epsilon_i-\epsilon_{i+1}$, 
$\Phi^+ := \{ \alpha_{ij} \,|\, 1 \leq i< j \leq 5\}$ 
(resp.\ $\Delta := \{ \alpha_1, \alpha_2, \alpha_3, \alpha_4\}$)  
is the set of positive (resp.\ simple) roots corresponding to the upper triangular Borel subgroup $B^-$. 
For 
$\omega_i:=
\epsilon_1+\cdots+\epsilon_i$, 
$\omega_1, \omega_2, \omega_3, \omega_4$ are the fundamental weights. 
The Weyl group $W (\simeq S_5)$ acts on $X(T) = \Z^5/\Z\delta$ by permuting the coordinates of $\Z^5$. 
Set 
\[
\rho := \frac{1}{2}\sum_{\alpha \in \Phi^+} \alpha = 
[4, 3, 2, 1, 0] = \omega_1 + \omega_2 + \omega_3+\omega_4. 
\]
\item 
Given $\lambda \in X(T)$, 
there is the corresponding line bundle $\cL_{\lambda}$ on
$\SL_5/B$. 
% For $\delta := (1, 1, 1, 1, 1) \in \Z^5 = X(T)$, 
% we have an isomorphism $\cL_{\lambda} \simeq 
% \cL_{\lambda +n\delta}$ of coherent sheaves 
% for every $n \in \Z$. 
% For $\omega_i
% :=
% \epsilon_1+\cdots+\epsilon_i$, 
%For $(a_1,\ldots,a_5) \in \Z^5$, we set $\cL_{(a_1, ..., a_5)} := \cL_{[a_1, ..., a_5]}$. 
The following are equivalent for $a_1, ..., a_5 \in \Z$:  
\begin{itemize}
\item $[a_1,\ldots,a_5]$ is dominant. 
\item $a_1\geq a_2\geq a_3\geq a_4\geq a_5$. 
\item $\cL_{[a_1,\ldots,a_5]}$ is nef. 
\item $[a_1,\ldots,a_5] =\sum_{i=1}^4 m_i \omega_i$ 
for some $m_1, ..., m_4 \in \Z_{\geq 0}$. 
\end{itemize}
\item
For $1\leq r\leq 4$, let 
$\pi_r:F=\SL_5/B\longrightarrow \Gr(r,5)$ 
be the natural projection. 
Then it is well known that $\pi_r^*\MO_{\Gr(r, 5)}(1) \simeq 
\cL_{\omega_{5-r}}$. 
In particular,  we have
$\pi^*_2\MO_{\Gr(2, 5)}(1)
\simeq
\cL_{\omega_3}
=
\cL_{[1, 1, 1, 0, 0]}$. 
\item 
For $F:=\SL_5/B$ and $\Sigma := \Gr(2,5)$, 
let $\pi:F=\SL_5/B\to \Sigma = \Gr(2, 5)$ be the natural projection. 
\end{enumerate}
\end{nota}

\begin{lem}\label{l genus 6 line bundle pushforwards}
We use Notation \ref{n weight lattice on GL5}. 
Let $\cU$ and $\cQ$ be the universal subbundle and quotient bundle on $\Sigma$, respectively. 
Then the following hold:
\begin{enumerate}\item 
% $\cL_{\omega_2-2\omega_3+\omega_4} \simeq \cL_{[0,0,-1,1,0]}$ 
% and 
$
R\pi_*\cL_{\omega_2-2\omega_3+\omega_4}
\simeq
\pi_*\cL_{\omega_2-2\omega_3+\omega_4}
\simeq
\cU\otimes\cQ^\vee \simeq \Omega^1_\Sigma.$ 
\item 
%$\cL_{2\omega_2-2\omega_3} \simeq \cL_{[0,0,-2,0,0]}$ and 
$
R\pi_*\cL_{2\omega_2-2\omega_3}
\simeq
\pi_*\cL_{2\omega_2-2\omega_3}
\simeq
\Sym^2\cQ^\vee.$ 
\item 
%$\cL_{\omega_1-2\omega_3+2\omega_4} \simeq \cL_{[1,0,0,2,0]}$ and 
$
R\pi_*\cL_{\omega_1-2\omega_3+2\omega_4}
\simeq
\pi_*\cL_{\omega_1-2\omega_3+2\omega_4}
\simeq
\Sym^2\cU\otimes\cQ.$
\end{enumerate}
%Moreover, the three line bundles appearing above are $\pi$-nef.
\end{lem}

\begin{proof}
% For each statement of (1), (2), (3), 
% the former isomorphism is ensured by direct computation. 
% Let us prove the latter ones. 
We have the following commutative diagram in which 
the lower two squares are cartesian.   
\[
\begin{tikzcd}[column sep=0.3em, row sep=2.2em]
&
&
F=\Fl(1,2,3,4;5)
\arrow[dl]
\arrow[dr]
\arrow[lldd, bend right, "p_3"']
\arrow[dd, "p_1"]
\arrow[rrdd, bend left, "p_4"]
&
\\
&
\Fl(1,2,3;5)
\arrow[dl]
\arrow[dr]
\arrow[dd, "q_{13}"]
%\arrow[ddr, phantom, "\lrcorner", very near start]
&
&
\Fl(1,2,4;5)
\arrow[dl]
\arrow[dr]
\arrow[dd, "q_{14}"]
%\arrow[ddl, phantom, "\lrcorner", very near start]
\\
E_3 := \Fl(2,3;5)
\arrow[dr, "q_3"']
&
&
E_1 :=\Fl(1,2;5)
\arrow[dl, "q_1"]
\arrow[dr, "q_1"']
&
&
E_4 :=\Fl(2,4;5)
\arrow[dl, "q_4"]
\\
&
\Gr(2, 5) \arrow[r, equal]
&
\Sigma \arrow[r, equal]
&
\Gr(2, 5)
&
\end{tikzcd}
\]
Note that 
\begin{equation}\label{e1 genus 6 line bundle pushforwards}
E_1\simeq \P_\Sigma(\cU),\qquad
E_3\simeq \P_\Sigma(\cQ^\vee),\qquad
E_4\simeq \P_\Sigma(\cQ),
\end{equation}
and each $q_i$ is the induced projective space bundle. 
For an integer $a \geq 0$ and the tautological divisor $\xi_i$ of $q_i$, 
we obtain %it holds that %the following hold for every $a\geq0$: 
\[
Rq_{1*}\MO_{E_1}(a\xi_1)\simeq \Sym^a\cU,\quad 
Rq_{3*}\MO_{E_3}(a\xi_3)\simeq \Sym^a\cQ^\vee,\quad 
Rq_{4*}\MO_{E_4}(a\xi_4)\simeq \Sym^a\cQ.
\]
As the above lower squares are cartesian, 
the following hold for all $a_1,a_3,a_4\geq0$: 
%the projection formula yields
\[
Rq_{13*}\MO_{\Fl(1,2,3;5)}(a_1\xi_1+a_3\xi_3)
\simeq
\Sym^{a_1}\cU\otimes_{\MO_\Sigma}\Sym^{a_3}\cQ^\vee, 
\]
\[
Rq_{14*}\MO_{\Fl(1,2,4;5)}(a_1\xi_1+a_4\xi_4)
\simeq
\Sym^{a_1}\cU\otimes_{\MO_\Sigma}\Sym^{a_4}\cQ.
\]
The three line bundles $\cL_{b_1\omega_1+ b_2 \omega_2 + b_3 \omega_3 + b_4 \omega_4}$ on $F$ appearing in the statement are $\pi$-nef, 
because the coefficients $b_1, b_2, b_4$ of $\omega_1, \omega_2, \omega_4$ are non-negative. 
Then it is enough to prove 
\begin{equation}\label{e2 genus 6 line bundle pushforwards}
\xi_1
\sim
\omega_4-\omega_3,
\qquad
\xi_3
\sim
\omega_2-\omega_3,
\qquad
\xi_4
\sim
\omega_1, 
\end{equation}
where we identify the fundamental weights with the corresponding divisor classes on the flag varieties.

By the canonical divisor formula for flag varieties, we have
\begin{equation}\label{e3 genus 6 line bundle pushforwards}
-K_{E_1}
\sim
2\omega_4+4\omega_3,
\qquad
-K_{E_3}
\sim
3\omega_2+3\omega_3,
\qquad
-K_{E_4}
\sim
3\omega_1+4\omega_3.
\end{equation}
Here 
% On the other hand, since
% \[
% E_1\simeq\P_\Sigma(\cU),
% \qquad
% E_3\simeq\P_\Sigma(\cQ^\vee),
% \qquad
% E_4\simeq\P_\Sigma(\cQ),
% \]
% the canonical divisor formula for projective bundles gives
it follows from (\ref{e1 genus 6 line bundle pushforwards}) that 
\begin{equation}\label{e4 genus 6 line bundle pushforwards}
-K_{E_1}
\sim
2\xi_1+6q_1^*H, \quad 
-K_{E_3}
\sim
3\xi_3+6q_3^*H,\quad 
-K_{E_4}
\sim
3\xi_4+4q_4^*H.
\end{equation}
As $H$ corresponds to $\omega_3$, 
(\ref{e2 genus 6 line bundle pushforwards}) holds by comparing 
(\ref{e3 genus 6 line bundle pushforwards}) and 
(\ref{e4 genus 6 line bundle pushforwards}). 
\qedhere
\end{proof}

\begin{lem}\label{l genus 6 Grassmannian Omega1}
Set $\Sigma:=\Gr(2,5)$.
Then the following hold: 
\begin{enumerate}
\item
$H^i(
\Sigma,
\Omega_\Sigma^1(m)
)
=0$ for every $i>0$ and $m>0$. 
\item
$H^i(
\Sigma,
\Omega_\Sigma^1(m)
)
=0$ for  every $i\geq0$ and $m\in\{-1,-2\}$. 
\end{enumerate}
\end{lem}

\begin{proof}
The assertion (1) follows from
\cite[Proposition 8.2]{KT26}.

Let us prove (2). 
Fix $i\geq 0$ and $m \in \{-1, -2\}$. 
In what follows, we use Notation \ref{n weight lattice on GL5}. 
By Lemma \ref{l genus 6 line bundle pushforwards}, 
the following holds for $\lambda_m := [0,0,-1,1-m,-m]$: 
%we get 
\[
H^i(\Sigma, \Omega^1_\Sigma(m)) 
\simeq H^i(F, \cL_{\omega_2 +(m-2)\omega_3+\omega_4})
\simeq 
H^i(F, \cL_{\lambda_m}). 
\]
%for every $m \in \Z$. 
Suppose that $H^i(F, \cL_{\lambda_m}) \neq 0$. 
It suffices to derive a contradiction. 
Let $\lambda^+_m$ denote the unique weight such that $\lambda^+_m +\rho$ is dominant 
and $\lambda^+_m = w\cdot \lambda_m$ for some $w \in W$. 
We can find a dominant weight $\mu \in X(T)$ such that 
$L(\mu)$ is a simple composition factor of $H^i(F, \cL_{\lambda_m})$. 
Then 
the strong
linkage principle (Theorem \ref{t SLP}) 
implies 
$\mu\leq\lambda^+_m$, that is,
\[
\lambda^+_m-\mu
=
\sum_{i=1}^4c_i\alpha_i
\]
for some $c_1,c_2,c_3,c_4\in\mathbb Z_{\geq0}$.
By direct computation, we obtain the following table:
\[
\begin{array}{c|c|c|c}
m & \lambda_m+\rho & \lambda_m^++\rho & \lambda_m^+\\
\hline
-1
&[4,3,1,3,1]
&[4,3,3,1,1]
&[0,0,1,0,1]= -\omega_2 + \omega_3-\omega_4
\\[2pt]
-2
&[4,3,1,4,2]
&[4,4,3,2,1]
&[0,1,1,1,1] =-\omega_1.
\end{array}
\]
Set
$\eta
:=
\alpha_1^\vee+\alpha_2^\vee+\alpha_3^\vee+\alpha_4^\vee$. 
By $\langle \omega_i,\eta\rangle=1$ for every $i\in\{1,2,3,4\}$, 
we get $\langle\lambda_m^+,\eta\rangle=-1$. %for $r\in\{1,2\}$.
On the other hand, 
we have  $\la \mu, \eta \ra \geq 0$ as $\mu$ is dominant. 
By using the fact that $( \la \alpha_i, \alpha_j^\vee\ra)$
 is the Cartan matrix of type $A_4$, we see that $\la \alpha_i, \eta \ra \geq 0$ for every $i \in \{1,2,3,4\}$. 
To summarise, we obtain the following contradiction: 
\[
-1 
= \langle\lambda_m^+,\eta\rangle 
=\langle \mu + \sum_{i=1}^4 c_i\alpha_i,\eta\rangle 
=\la \mu, \eta\rangle + \sum_{i=1}^4 c_i \la \alpha_i,  \eta\ra \geq 0+\sum_{i=1}^4 c_i \cdot 0
=0. 
\]

\qedhere

\end{proof}

\begin{lem}\label{l genus 6 Grassmannian Omega2}
Assume $p \neq 2$. 
Set $\Sigma:=\Gr(2,5)$. 
%and let $H$ be the Pl\"ucker divisor.
%Assume that $p\neq2$.
Then $H^i(
\Sigma,
\Omega_\Sigma^2(m)
)
=0$ for every $i>0$ and every $m>0$.
\end{lem}

\begin{proof}
For the universal subbundle $\cU$ (resp.\ quotient bundle $\cQ$) on $\Sigma = \Gr(2, 5)$, 
the following hold: 
\begin{itemize}
\item $0
\to
\cU
\to
\cO_\Sigma^{\oplus5}
\to
\cQ
\to
0$. 
\item $\rank\cU=2$ and $\rank\cQ=3$.
\item $\Omega_\Sigma^1\simeq\cU\otimes\cQ^\vee.$ 
\item $\cU^\vee \simeq \cU(1)$ and $\wedge^2 \cQ^\vee \simeq \cQ(-1)$. 
\end{itemize}
Since $p\neq2$, we obtain
\[
\Omega_\Sigma^2
\simeq
\wedge^2\cU\otimes\Sym^2\cQ^\vee \oplus
\Sym^2\cU\otimes\wedge^2\cQ^\vee. 
\]
Hence it is enough to show (1) and (2) below. 
\begin{enumerate}
\item $H^{>0}(\Sigma, \wedge^2\cU\otimes\Sym^2\cQ^\vee (m))=0$ for every $m>0$. 
\item $H^{>0}(\Sigma, \Sym^2\cU\otimes\wedge^2\cQ^\vee (m))=0$ for every $m>0$. 
\end{enumerate}

Let us show (1). 
As $\wedge^2\cU\simeq\cO_\Sigma(-1)$, we have
\[
\wedge^2\cU\otimes\Sym^2\cQ^\vee(m)
\simeq
\Sym^2\cQ^\vee(m-1).
\]
We use the same notation as in the proof of Lemma \ref{l genus 6 line bundle pushforwards}. 
Then $-K_{E_3} \sim 3\omega_2 + 3 \omega_3$ and 
\[
H^{>0}(\Sigma, \wedge^2\cU\otimes\Sym^2\cQ^\vee(m)) 
\simeq 
H^{>0}(\Sigma, \Sym^2\cQ^\vee(m-1)) 
\overset{{\rm (i)}}{\simeq} 
H^{>0}(F, \cL_{2\omega_2 +(m-3)\omega_3}) 
\]
\[
\simeq 
H^{>0}(E_3, \cL_{2\omega_2 +(m-3)\omega_3}) 
\overset{{\rm (ii)}}{=} 0, 
\]
where (i) follows from Lemma \ref{l genus 6 line bundle pushforwards} and (ii) holds by the fact that 
the divisor $2\omega_2 +(m-3)\omega_3 -K_{E_3} \sim 5\omega_2 +m\omega_3$ on $E_3$ is ample. 
Thus (1) holds.

Let us show (2). 
As $\wedge^2\cQ^\vee(1)\simeq\cQ$, we have 
\[
\Sym^2\cU\otimes\wedge^2\cQ^\vee (m) \simeq \Sym^2\cU\otimes \cQ(m-1). 
\]
For $F' := \Fl(1, 2, 4; 5)$, we have 
$-K_{F'} \sim 3\omega_1 + 3\omega_3  + 2 \omega_4$ and 
\[
H^{>0}(\Sigma,\Sym^2\cU\otimes\wedge^2\cQ^\vee (m)) 
\simeq 
H^{>0}(\Sigma, \Sym^2\cU\otimes \cQ(m-1)) 
\overset{{\rm (iii)}}{\simeq} 
H^{>0}(F, \cL_{\omega_1 +(m-3)\omega_3+2\omega_4}) 
\]
\[
\simeq 
H^{>0}(F', \cL_{\omega_1 +(m-3)\omega_3+2\omega_4}) 
\overset{{\rm (iv)}}{=} 0, 
\]
where (iii) follows from Lemma \ref{l genus 6 line bundle pushforwards} and (iv) holds by the fact that 
the divisor $\omega_1 +(m-3)\omega_3+2\omega_4 -K_{F'} 
\sim 4\omega_1 +m\omega_3 +4\omega_4$ on $F'$ is ample. 
Thus (2) holds. 
\end{proof}

\subsection{Complete-intersection case}\label{ss g=6 CI}

We first treat the complete-intersection case.

\begin{prop}\label{p genus 6 ordinary}
Assume $p>3$. Let $\Sigma := \Gr(2, 5) \subset \P^9$ be the Pl\"{u}cker embedding. 
Let $X$ be a prime Fano threefold of genus $6$ 
which is a complete intersection 
$\Sigma\cap H_1\cap H_2\cap Q$, 
where  $H_1$ and $H_2$ are hyperplanes and $Q$ is a quadric hypersurface on $\P^9$. 
Then
$H^1(
X,
\Omega_X^2(pH)
)
=0.$ 
\end{prop}

\begin{proof}
%We identify the restriction of $H$ to $X$ with $H$.
Set $\cE:=
\cN_{X/\Sigma}^{\vee}
\simeq
\cO_X(-1)^{\oplus2}
\oplus
\cO_X(-2)$. 
Applying $\wedge^2$ to the conormal exact sequence 
$0
\to
\cE
\to
\Omega_\Sigma^1|_X
\to
\Omega_X^1
\to
0,$ we get the following exact sequences for some vector bundle $\cF$:
\[
0
\to
\cF
\to
\Omega_\Sigma^2|_X
\to
\Omega_X^2
\to
0, \quad 
0
\to
\wedge^2\cE
\to
\cF
\to
\cE\otimes\Omega_X^1
\to
0.
\]
Hence it is enough to prove (i)-(iii) below. 
\begin{enumerate}
\item[(i)] $H^1(X, \Omega^2_\Sigma(p)|_X)=0$. 
\item[(ii)] $H^2(X, \wedge^2 \cE(p))=0$. 
\item[(iii)] $H^2(X, \cE \otimes \Omega^1_X(p))=0$. 
\end{enumerate}
By using $\wedge^2 \cE \simeq \MO_X(-2) \oplus \MO_X(-3)^{\oplus 2}$, 
the assertions (ii) and (iii) follow from 
Theorem \ref{t Fano3 ANV} and 
Proposition \ref{p quasi-F-split quadrics}, respectively.

Let us show (i). 
The Koszul resolution of $X\subset \Sigma$ is given as follows: 
\[
0
\to
\cO_\Sigma(-4)
\to
 \MO_{\Sigma}(-2) \oplus \MO_{\Sigma}(-3)^{\oplus 2}
\to
\cO_\Sigma(-1)^{\oplus2}
\oplus
\cO_\Sigma(-2)
\to
\cO_\Sigma
\to
\cO_X
\to
0.
\]
By taking the decomposition into the corresponding short exact sequences, it is enough to prove
\[
\begin{aligned}
&H^1(\Sigma,\Omega_\Sigma^2(p))=0,\\
&H^2(\Sigma,\Omega_\Sigma^2(p-1))
=
H^2(\Sigma,\Omega_\Sigma^2(p-2))=0,\\
&H^3(\Sigma,\Omega_\Sigma^2(p-2))
=
H^3(\Sigma,\Omega_\Sigma^2(p-3))=0,\\
&H^4(\Sigma,\Omega_\Sigma^2(p-4))=0.
\end{aligned}
\]
Since $p > 4$, all the twists appearing above are positive, 
and hence these cohomologies vanish by  Lemma \ref{l genus 6 Grassmannian Omega2}. 
Thus (i) holds. 
\end{proof}

\subsection{Double-cover case}\label{ss g=6 2-to-1}

We next treat the double-cover case.
We first establish the vanishings on the del Pezzo threefold of degree $5$
that will be used below.

\begin{lem}\label{l genus 6 V5}
Assume $p \neq 2$. 
Let $V$ be the smooth quintic del Pezzo threefold.
Then the following hold:
\begin{enumerate}
\item 
$H^1(V,\cO_V(m))= H^2(V, \MO_V(m))=0$ for every $m \in \Z$ and 
$H^3(V, \MO_V(m))=0$ for every $m \geq -1$. 
\item 
$H^i(V,\Omega_V^1(m))=0$ for every $i>0$ and $m \geq 1$. 
% one of the following holds: 
% \begin{itemize}
% \item $i=1$ and $m\geq 2$. 
% \item $i\in\{2,3\}$ and $m\geq1$. 
% \end{itemize}
\item
$H^i(V,\Omega_V^2(m))=0$ for $i >0$ and $m\geq3$.
\end{enumerate}
\end{lem}

\begin{proof}
Let $\Sigma := \Gr(2, 5) \subset \P^9$ denote the Pl\"{u}cker embedding. 
We may assume 
\[
V=
\Sigma\cap H_1\cap H_2\cap H_3,
\]
for some hyperplanes $H_1, H_2, H_3$ on $\Sigma$. 
The assertion (1) follows from the Kodaira vanishing theorem for $V$ (Theorem \ref{t Fano3 ANV}).

Let us show (2). 
Fix $i>0$ and $m\geq 1$. 
%$m$ and $i$ as in the statement. 
By $m \geq 1$, (1) implies  $H^{>0}(V, \MO_V(m-1))=0$. 
By the conormal exact sequence
\[
0
\to
\cO_V(-1)^{\oplus3}
\to
\Omega_\Sigma^1|_V
\to
\Omega_V^1
\to
0, 
\]
it suffices to show $H^i(V, \Omega_\Sigma^1(m)|_V) =0$. 
%$H^i(V, \Omega_\Sigma^1(m)|_V) \xrightarrow{\simeq} H^i(V, \Omega_V^1(m))$. 
We have the Koszul exact sequence 
\[
0 
\to \MO_\Sigma(-3) 
\to \MO_\Sigma(-2)^{\oplus 3} 
\to \MO_\Sigma(-1)^{\oplus 3} \to  \MO_\Sigma \to \MO_V \to 0. 
\]
Hence the problem is reduced to showing 
\[
H^{i+\ell}(\Sigma, \Omega_\Sigma^1(m-\ell)) = 0
\]
for every $\ell \in \{0, 1, 2, 3\}$. 
Since $i+\ell >0$, 
this follows from Lemma \ref{l genus 6 Grassmannian Omega1} when $m- \ell \neq 0$. 
If $m -\ell =0$, then we have $i+\ell =i+m \geq 2$, and hence 
the equality $H^{i+\ell}(\Sigma, \Omega_\Sigma^1(m-\ell)) = 0$  holds by 
Remark \ref{r non-diag vanishing}. 
Thus (2) holds.

Let us show (3). 
Fix $i >0$ and $m \geq 3$. 
Applying $\wedge^2$ to the conormal exact sequence, we obtain
the following exact sequences for some vector bundle $\cF$:
\[
0
\to
\cF
\to
\Omega_\Sigma^2|_V
\to
\Omega_V^2
\to
0, \qquad 
0
\to
\cO_V(-2)^{\oplus3}
\to
\cF
\to
\Omega_V^1(-1)^{\oplus3}
\to
0.
\]
By (1) and (2), we get $H^{>0}(V, \MO_V(m-2))= H^{>0}(V, \Omega_V^1(m-1))=0$, 
which implies $H^{>0}(V, \cF(m))=0$. 
Hence it suffices to show $H^i(V, \Omega^2_\Sigma(m)|_V)=0$. 
By the above Koszul exact sequence, the problem is reduced to showing 
\[
H^{i+\ell}(\Sigma, \Omega^2_\Sigma(m-\ell))=0
\]
for every $\ell \in \{0, 1, 2, 3\}$. 
If $m-\ell >0$, then this follows from Lemma \ref{l genus 6 Grassmannian Omega2}. 
Otherwise, we get $m=\ell =3$ and $i+\ell \geq 1 + 3=4$. 
In this case, the equality $H^{i+\ell}(\Sigma, \Omega^2_\Sigma(m-\ell))=0$ follows from 
$i+\ell \neq 2$ (Remark \ref{r non-diag vanishing}). 
Thus (3) holds. 
\end{proof}

\begin{prop}\label{p genus 6 special}
Assume $p>3$. 
Let $X$ be a prime Fano threefold of genus $6$ which has a separable double cover
\[
f\colon X\to V
\]
onto the smooth quintic del Pezzo threefold $V$.
Then $H^1(
X,
\Omega_X^2(p)
)
=0.$ 
\end{prop}

\begin{proof}
% {\cred Do we need the 1st paragraph?}
% Let $L$ be the ample generator of $\Pic(V)$.
% We have $-K_V\sim2L$. 
% Since $p>3$, the degree of $f$ is invertible in $k$.
% The induced $\MO_V$-module homomorphism $\cO_V\to f_*\cO_X$, 
% and the standard description of the double cover gives
% $f_*\cO_X
% \simeq
% \cO_V\oplus\cO_V(-L)$. 
% Moreover,
% $K_X
% \sim
% f^*(K_V+L)
% \sim
% -f^*L,$ 
% and hence
% $H=f^*L$. 
By Lemma \ref{l genus 6 V5}, we have
\begin{enumerate}
\item 
$H^1(V,\cO_V(p-m))=0$ for $m\in\{2,3\}$. 
\item 
$H^2(V,\cO_V(p-m))=0$ for $m\in\{1,2,3,4\}$. 
\item 
$H^3(V,\cO_V(p-m))=0$ for $m\in\{1,2,3,4,5\}$. 
\item 
$H^1(V,\Omega_V^1(p-m))=0$ for $m\in\{1,2,3\}$, 
\item 
$H^2(V,\Omega_V^1(p-m))=0$ for $m\in\{2,3,4\}$. 
\item 
$H^3(V,\Omega_V^1(p-4))=0$. 
\item 
$H^1(V,\Omega_V^2(p-n))=0$ for $n\in\{0,1,2\}$. 
\item 
$H^2(V,\Omega_V^2(p-2))=0.$ 
\end{enumerate}
Thus \cite[Proposition 6.21]{KTLift2} is applicable, and hence 
$H^1(
X,
\Omega_X^2(p)
)
=0.$ 
\end{proof}

\begin{thm}\label{t genus 6 F-split}
Assume $p>3$. 
Let $X$ be a prime Fano threefold of genus $6$. 
%over an algebraicallyclosed field of characteristic $p>3$.
Then $X$ is $F$-split.
\end{thm}

\begin{proof}
%By {\cred ref}, 
By \cite[Theorem 5.5]{KTLift1}, 
Proposition \ref{p genus 6 ordinary} or
Proposition \ref{p genus 6 special} is applicable, and hence 
$H^1(
X,
\Omega_X^2(p)
)
=0$. 
Then  $X$ is $F$-split by Proposition \ref{p quasi-F-split quadrics}. 
%The $F$-splitting criterion therefore implies that $X$ is $F$-split.
\end{proof}

\section{Genus 7}

The purpose of this section is to prove that a prime Fano threefold of genus $7$ is $F$-split when $p>2$.

\begin{nota}\label{nota g=7}
Assume that $p\neq 2$.
Fix the following $10 \times 10$ matrix: 
\[
J:=
\begin{pmatrix}
O&E_5\\
E_5&O
\end{pmatrix} \in \Mat_{10}(k).
\]
\begin{enumerate}
\item 
Set $G:=\SO_{10}
:=
\left\{
g\in\SL_{10}
\mathrel{}\middle|\mathrel{}
g^tJg=J
\right\}$, which is known to be a semisimple algebraic group. 
Set 
\[
T
:=
\left\{
\begin{pmatrix}
D&O\\
O&D^{-1}
\end{pmatrix}
\mathrel{}\middle|\mathrel{}
D=\diag(t_1,\ldots,t_5),\
t_i\in k^\times
\right\}, 
\]
\[
B
:=
\left\{
\begin{pmatrix}
A&AX\\
O&(A^{-1})^t
\end{pmatrix}
\mathrel{}\middle|\mathrel{}
A\in B_{\GL_5},\
X\in\Alt_5(k)
\right\}, 
\]
where $\Alt_5(k)
:=
\left\{
X\in\Mat_5(k)
\mathrel{}\middle|\mathrel{}
X^t=-X
\right\}$ and 
$B_{\GL_5}$ denotes  the lower triangular Borel subgroup of $\GL_5$.
Then $T$ is a maximal torus and $B$ is a Borel subgroup of $\SO_{10}$ satisfying $\SO_{10} \supset B \supset T$. 
\item 
Set 
 $W:=k^5\oplus 0\subset k^5\oplus k^5 =k^{10}$, 
 which is  a maximal isotropic subspace of $k^{10}$ 
 with respect to the non-degenerate symmetric bilinear form
$(v,w) \mapsto v^tJw$ induced by $J$. 
The stabiliser of $W$ is the maximal parabolic subgroup
\[
P:=\Stab_{\SO_{10}}(W)
=
\left\{
\begin{pmatrix}
A&AX\\
O&(A^{-1})^t
\end{pmatrix}
\mathrel{}\middle|\mathrel{}
A\in\GL_5,\ 
X\in\Alt_5(k)
\right\}.
\]
We have $P=L\ltimes R_u(P)$ for the  unipotent radical
\[
R_u(P)
=
\left\{
\begin{pmatrix}
I_5&X\\
O&I_5
\end{pmatrix}
\mathrel{}\middle|\mathrel{}
X\in\Alt_5(k)
\right\},
\]
and the Levi subgroup 
\[
L
:=
\left\{
\begin{pmatrix}
A&O\\
O&(A^{-1})^t
\end{pmatrix}
\mathrel{}\middle|\mathrel{}
A\in\GL_5
\right\}
\simeq\GL_5
\]
of $P$ containing $T$. Moreover, we have 
\[
B_L := B\cap L
=
\left\{
\begin{pmatrix}
A&O\\
O&(A^{-1})^t
\end{pmatrix}
\mathrel{}\middle|\mathrel{}
A\in B_{\GL_5}
\right\}
\simeq B_{\GL_5}.
\]
For $F:= G/B$ and $\Sigma := G/P$, let 
$\pi:F=G/B\to \Sigma = G/P$ be the natural projection. 
\item 
For $1\leq i\leq5$, let $\epsilon_i\in X(T)$ be the character
defined by
\[
\epsilon_i\bigl(\diag
(t_1,\ldots,t_5,t_1^{-1},\ldots,t_5^{-1})\bigr)
:=t_i.
\]
Set $\nu :=(2, 1, 1, 0, 0) := 2\epsilon_1 + \epsilon_2 + \epsilon_3 \in X(T)$ and let $\lambda \in X(B)$ be the unique extension of $\nu$. 
\end{enumerate}
\end{nota}

\begin{rem}\label{r g=7 fund wt}
We use Notation \ref{nota g=7}. 
Set
\[
e_1:=-\epsilon_5,\qquad
e_2:=-\epsilon_4,\qquad
e_3:=-\epsilon_3,\qquad
e_4:=-\epsilon_2,\qquad
e_5:=\epsilon_1.
\]
Then
\[
\alpha_i=e_i-e_{i+1}\quad(1\leq i\leq4),
\qquad
\alpha_5=e_4+e_5
\]
form a set of simple roots for which $B$ is the negative Borel
subgroup and $P$ is the reduced maximal parabolic subgroup
corresponding to $\alpha_4$.
The corresponding fundamental weights are
\[
\omega_1=e_1,\qquad
\omega_2=e_1+e_2,\qquad
\omega_3=e_1+e_2+e_3,
\]
\[
\omega_4=
\frac{e_1+e_2+e_3+e_4-e_5}{2},
\qquad
\omega_5=
\frac{e_1+e_2+e_3+e_4+e_5}{2}.
\]
Here $\omega_4$ and $\omega_5$ belong to
$X(T)\otimes_{\mathbb Z}\mathbb Q$, but not to $X(T)$.
We have 
\begin{equation}\label{eq g=7 nu ep-to-omega}
\nu =(2, 1, 1, 0, 0) = 2\epsilon_1 + \epsilon_2 + \epsilon_3 = 2e_5-e_4 -e_3 = \omega_2 + \omega_5 -3\omega_4. 
\end{equation}
\end{rem}

\begin{rem}
We use Notation \ref{nota g=7}. 
Then it is well known that 
$\Sigma$ is isomorphic to a connected component 
$\OGr^+(5, 10)$ of the orthogonal Grassmannian variety 
$\OGr(5, 10)$ 
parametrising the $5$-dimensional isotropic subspaces of $k^{10}$. 
For the universal subbundle $\cU$ on $\Sigma$, 
it is well known that 
\[
\dim \Sigma=10,\qquad
\det\mathcal U\simeq\mathcal O_\Sigma(-2),\qquad
\Omega_\Sigma^1 \overset{(\star)}{\simeq}\wedge^2\mathcal U,\qquad
\omega_\Sigma^{-1} \simeq \MO_{\Sigma}(8), 
\]
where $(\star)$ holds by the proof of 
\cite[Lemma 2.3]{MR24}. 
For the closed embedding $\Sigma \subset \P^{15}$ induced by $|\MO_{\Sigma}(1)|$, 
an arbitrary prime Fano threefold $X$ of genus $7$ is 
isomorphic to $\Sigma \cap \P^8$ for some $8$-dimensional linear subvariety $\P^8 \subset \P^{15}$ \cite[Theorem 1.2]{KTprime}. 
\end{rem}

\begin{lem}\label{l genus 7 unified v1 Schur}
Assume that $p>2$.
Let $B_{\GL_5}\subset\mathrm{GL}_5$ be the lower triangular Borel subgroup and let $U$ be the standard $5$-dimensional $\GL_5$-module.
Then the following hold: 
\begin{enumerate}
\item 
$H^0(\mathrm{GL}_5/B_{\GL_5},\cL_{(2,1,1,0,0)})$ 
is a simple $\mathrm{GL}_5$-module. 
\item There is a $\GL_5$-module isomorphism 
\[
\wedge^2(\wedge^2U)
\simeq
H^0(\mathrm{GL}_5/B_{\GL_5},\cL_{(2,1,1,0,0)}). 
\]
\end{enumerate}
\end{lem}

\begin{proof}
Set $\nu:=(2,1,1,0,0)$.

Let us show (1). 
Suppose that there exists a simple composition factor $L(\mu)$ of 
$H^0(\mathrm{GL}_5/B_{\GL_5},\cL_\nu)$ of highest weight $\mu =(\mu_1, \mu_2, \mu_3, \mu_4, \mu_5) \in \Z^5$ with $\mu \neq \nu$. 
Since the weight $\nu$ occurs with multiplicity one, it is enough to derive a contradiction. 
Since the $\mathrm{GL}_5$-module
$H^0(\mathrm{GL}_5/B_{\GL_5},\cL_\nu)$ is a homogeneous polynomial representation of degree $4$, 
we get 
\begin{enumerate}
\item[(i)] $\sum_{i=1}^5 \mu_i =4$ and $\mu_1 \geq \mu_2 \geq \mu_3 \geq \mu_4 \geq \mu_5 \geq 0$. 
\end{enumerate}
As $\nu$ is dominant, so is $\nu + \rho$, and hence $\nu^+ = \nu$. 
We have $\rho = (2, 1, 0, -1, -2)$ by direct computation. 
By the strong linkage principle (Theorem \ref{t SLP}), 
the following hold: 
\begin{enumerate}
\item[(ii)] $\mu \leq \nu$, and 
\item[(iii)] 
$\nu + \rho = w(\mu + \rho)+ p\alpha$ 
for some $w \in S_5$ and $\alpha \in \Z \Phi \subset \Z^5$. 
%$\sum_{i=1}^5 \mu_i =4$ and $\mu_i\geq 0$ for every $i$. 
\end{enumerate}
As $\mu \neq \nu$, (i) and (ii) imply 
$\mu=(1,1,1,1,0)$.
%$\mu\leq\nu$, and 
%The only partitions satisfying this dominance inequality are
We then have 
%Then
\[
\nu+\rho=(4,2,1,-1,-2),
\qquad
\mu+\rho=(3,2,1,0,-2).
\]
However, these and (iii) lead to the following contradiction: 
\[
p\Z \ni (4^2 +2^2 + 1^2 +(-1)^2 +(-2)^2) - 
(3^2 +2^2 +1^2 + 0^2 +(-2)^2)= 8. 
\]
Thus (1) holds.

Let us show (2). We have 
\[
\dim \wedge^2(\wedge^2U)
=
\binom{\dim\wedge^2U}{2}
=
\binom{10}{2}
=
45.
\]
As $\nu$ is dominant, 
the Weyl dimension formula \cite[Corollary II.5.11]{Jan03} implies 
\[
\dim H^0(\mathrm{GL}_5/B_{\GL_5},\cL_\nu) 
%= \chi(\mathrm{GL}_5/B_{\GL_5},\cL_\nu) 
=
\prod_{1\leq i<j\leq5}
\frac{\nu_i-\nu_j+j-i}{j-i}
=
45.
\]
By (1), it is enough to show that $\nu$ is the largest weight of 
$\wedge^2(\wedge^2U)$. 
Let $U = \bigoplus_{i=1}^5 k u_i$ be the weight decomposition, 
where $u_1, %u_2, 
..., u_5$ are of weight 
$(1, 0, 0, 0, 0), %(0, 1, 0, 0, 0), 
..., (0, 0, 0, 0, 1)$. 
By using the weight decomposition 
\[
\wedge^2(\wedge^2 U) = 
\bigoplus_{(a, b, c, d) \in I} k (u_a \wedge u_b) \wedge (u_c \wedge u_d)
\]
for a suitable subset $I \subset \Z^4$, 
we see that $(u_1 \wedge u_2) \wedge (u_1 \wedge u_3) \in \wedge^2(\wedge^2 U)$ is of weight $\nu =(2, 1, 1, 0, 0)$ and 
$\nu$ is the largest weight of $\wedge^2(\wedge^2 U)$. 
Thus (2) holds. 
\end{proof}

\begin{lem}\label{l g=7 Omega2 lb}
We use Notation \ref{nota g=7}. 
Then $\Omega_\Sigma^2 \simeq \pi_*\cL_\lambda.$ 
\end{lem}

\begin{proof}
% We may assume that $B$ is the lower triangular Borel subgroup of $\SO_{10}$. 
Note that $B_L$ is the lower triangular
Borel subgroup of $\GL_5$ as in
Lemma \ref{l genus 7 unified v1 Schur}. 
We then get $L$-module isomorphisms
\[
\Omega_{\Sigma,o}^2
\simeq
\wedge^2(\wedge^2\mathcal U_o)
\simeq 
H^0(L/B_L,\cL_{\nu}),
\]
where the second isomorphism is ensured by Lemma \ref{l genus 7 unified v1 Schur} and 
the first isomorphism is $L$-linear, because
$\Omega_\Sigma^2
\simeq
\wedge^2(\wedge^2\mathcal U)$ is a $G$-equivariant isomorphism.
On the other hand, 
we have the following $L$-module isomorphisms for $o := eP \in G/P =\Sigma$: 
\[
\pi_*\cL_{\lambda}|_o
\simeq
H^0(P/B,\cL_{\lambda})
\simeq
H^0(L/B_L,\cL_{\nu}).
\]
Hence we get an $L$-module isomorphism 
$\theta :
\Omega^2_{\Sigma, o}
\xrightarrow{\simeq}
\pi_*\cL_{\lambda}|_o$ between the $P$-modules. 
As $\Omega^2_{\Sigma, o} \simeq H^0(L/B_L,\cL_{\nu})$ is a simple $L$-module (Lemma \ref{l genus 7 unified v1 Schur}), 
both sides of $\theta$ are simple $L$-modules, and hence simple $P$-modules. 
Then the actions of the unipotent radical $R_u(P)$ of $P$ 
are trivial \cite[Lemma 19.16]{Mil17}. 
By $P=L\ltimes R_u(P)$, $\theta :
\Omega^2_{\Sigma, o}
\to 
\pi_*\cL_{\lambda}|_o$ is a $P$-module isomorphism, 
which implies $\Omega_\Sigma^2 \simeq \pi_*\cL_\lambda.$ 
\qedhere
\end{proof}

\begin{prop}\label{p g=7 ambient Omega2}
Assume $p>2$ and set $\Sigma := \OGr^+(5, 10)$. 
Then 
\[
H^i(\Sigma,\Omega_\Sigma^2(m))=0
\]
if one of the following holds: 
\begin{enumerate}
\item
$m \geq 2$ and $ i \geq 1$. 
\item
$m=1$ and $i \geq 2$. 
\item
$m=0$ and $i \neq 2$.
\item
$m\in\{-1,-2,-3,-4\}$ and $i \geq 0$. 
\end{enumerate}
\end{prop}

\begin{proof}
The assertion (3) follows from Remark \ref{r non-diag vanishing}.

In what follows, we use Notation \ref{nota g=7} and 
we regard all weights as weights of the universal cover 
$\wt{G} = \Spin_{10}$ of $G = \SO_{10}$. 
Let $\omega_1, ..., \omega_5$ be 
the fundamental weights given in Remark \ref{r g=7 fund wt}. 
For $m \in \Z$, we set  
\[
\lambda(m) := \lambda + m\omega_4 = \omega_2+(m-3)\omega_4+\omega_5. 
\]
As 
$\cL_{\omega_4} \simeq \pi^*\MO_{\Sigma}(1)$ and $\cL_{\lambda(m)}$ is $\pi$-nef,  
it follows from Lemma \ref{l g=7 Omega2 lb} that 
$
R\pi_*\cL_{\lambda(m)} \simeq 
\pi_*\cL_{\lambda(m)} \simeq 
(\pi_*\cL_{\lambda})(m)
\simeq \Omega^2_\Sigma(m)$, 
which implies 
\begin{equation}\label{e genus 7 unified v1 flag cohomology}
H^i(\Sigma,\Omega_\Sigma^2(m))
\simeq H^i(F,\mathcal L_{\lambda(m)}).
\end{equation}
It holds that 
\[
\mathcal L_{\lambda(m)}\otimes\mathcal O_F(-K_F)
\simeq
\mathcal L_{2\omega_1+3\omega_2+2\omega_3+(m-1)\omega_4+3\omega_5}.
\]
Hence (1) and (2) follow from 
Proposition \ref{p flag Kodaira} and Proposition \ref{p almost ample}, respectively.

It suffices to prove (4). 
If $L(\mu)$ is a simple composition factor of
$H^i(F,\mathcal L_{\lambda(m)})$, then $\mu$ is dominant and the strong
linkage principle (Theorem \ref{t SLP}) gives
$\mu\leq\lambda(m)^+$, that is, 
\[
\lambda(m)^+-\mu=\sum_{j=1}^5c_j\alpha_j
\]
for some $c_1, ..., c_5 \in \Z_{\geq 0}$. 

In what follows, we use the identification $\R^5 
= \R e_1 \oplus \cdots \oplus \R e_5$ by setting $(x_1, ..., x_5) := x_1e_1 + \cdots + x_5e_5$ for $x_1, ..., x_5 \in \R$. 
We have 
%In standard coordinates $\R^5 = \R e_1 \oplus \cdots \oplus \Re_5$, 
\[
\lambda(m)= \lambda + m\omega_4 = (2e_5 -e_4-e_3) + m\cdot 
\frac{e_1+e_2+e_3+e_4-e_5}{2}
\]
\[
=
(
\frac m2,\frac m2,\frac{m-2}{2},
\frac{m-2}{2},\frac{4-m}{2}
).
\]
The Weyl group of type $D_5$ acts by permutations and an even number
of sign changes. 
A weight $(x_1,\ldots,x_5)$ is dominant if and
only if $x_1\geq x_2\geq x_3\geq x_4\geq|x_5|$. 
By $\rho=\omega_1+\cdots+\omega_5=(4, 3, 2, 1, 0)$,  
we obtain the following list. 
\[
\begin{array}{c|c|c|c}
m & \lambda(m)+\rho & \lambda(m)^++\rho & \lambda(m)^+\\
\hline
-1
& (\frac72,\frac52,\frac12,-\frac12,\frac52)
& (\frac72,\frac52,\frac52,\frac12,-\frac12)
& (-\frac12,-\frac12, *, *, *)\\[2pt]
-2
& (3,2, 0,-1,3)
& (3,3,2,1,0)
& (-1,0,*,*,*)\\[2pt]
-3
& (\frac52,\frac32,-\frac12,-\frac32,\frac72)
& (\frac72,\frac52,\frac32,\frac32,\frac12)
& (-\frac12,-\frac12,*,*,*)\\[2pt]
-4
& (2,1,-1,-2,4)
& (4,2,2,1,1)
& (0,-1,*,*,*)
\end{array}
\]
Define the linear function 
\[
f: \R^5 \to \R, \qquad f(x_1,\ldots,x_5):=x_1+x_2.
\]
For every dominant weight $\mu$, we have $f(\mu)\geq0$.
Moreover,
\[
f(\alpha_2)=1,
\qquad
f(\alpha_j)=0\quad(j\neq2),
\]
while the above table gives $f(\lambda(m)^+)=-1$ in all four cases.
Consequently, $\mu\leq\lambda(m)^+$ would imply
\[
-1 \geq -1-f(\mu)
=f(\lambda(m)^+)- f(\mu)
=f(\lambda(m)^+-\mu)
=\sum_{j=1}^5 c_j f(\alpha_j)
=c_2\geq0,
\]
which is absurd.
Hence $H^i(F,\mathcal L_{\lambda(m)})$ has no simple composition
factor and hence $H^i(F,\mathcal L_{\lambda(m)}) =0$. 
This, together with \eqref{e genus 7 unified v1 flag cohomology},  implies  (4).
\end{proof}

\begin{thm}\label{t genus 7 unified v1 F split}
Assume $p>2$. 
Let $X$ be a prime Fano threefold of genus $7$. 
%realized as a codimension-seven linear section of $\Sigma$.
Then $X$ is $F$-split.
\end{thm}

\begin{proof}
Recall that $X$ is a linear section of $\Sigma := \OGr^+(5, 10) \subset \P^{15}$ \cite[Theorem 1.2]{KTprime}. 
In particular, $\cN^{\vee}_{X/\Sigma} \simeq \MO_X(-1)^{\oplus 7}$. 
Applying $\wedge^2$ to the conormal exact sequence $0 \to \MO_X(-1)^{\oplus 7} \to \Omega_\Sigma^1|_X \to \Omega^1_X \to 0$, 
we get the following exact sequences for some vector bundle $\cF$: 
\[
0 \to \cF \to \Omega_\Sigma^2|_X \to \Omega^2_X \to 0, \qquad 
0 \to \MO_X(-2)^{\oplus 21} 
\to \cF \to \Omega^1_X(-1)^{\oplus 7} \to 0, 
\]
where we used $\wedge^2 (\MO_X(-1)^{\oplus 7})  \simeq \MO_X(-2)^{\oplus 21}$. 
It is enough to show $H^1(X, \Omega_X^2(p))=0$ (Proposition \ref{p quasi-F-split quadrics}). 
As  $H^2(X, \MO_X(p-2))= H^2(X, \Omega_X^1(p-1))=0$  (Proposition \ref{p quasi-F-split quadrics}), 
we get $H^2(X, \cF(p))=0$. 
Therefore, the problem is reduced to $H^1(X, \Omega_\Sigma^2(p)|_X)=0$. 
We have the Koszul exact sequence
\[
0
\to
\MO_\Sigma(-7)
\to
\MO_\Sigma(-6)^{\oplus 7}
\to
\MO_\Sigma(-5)^{\oplus 21}
\to
\MO_\Sigma(-4)^{\oplus 35}
\]
\[
\to
\MO_\Sigma(-3)^{\oplus 35}
\to
\MO_\Sigma(-2)^{\oplus 21}
\to
\MO_\Sigma(-1)^{\oplus 7}
\to
\MO_\Sigma
\to
\MO_X
\to0. 
\]
Applying $(-) \otimes \Omega^2_\Sigma(p)$ and the decomposition into the short exact sequences, 
it suffices to prove that 
\[
H^{\ell+1}(\Sigma, \Omega^2_\Sigma(p-\ell))=0
\text{ \qquad if \qquad}
0 \leq \ell \leq 7.
\]
If $p-\ell <0$ or $p-\ell \geq 2$, then the assertion follows from Proposition \ref{p g=7 ambient Omega2}.
If $p-\ell=1$, then $\ell+1=p\geq 2$, and hence the assertion follows from Proposition \ref{p g=7 ambient Omega2}.
Thus we may assume $p-\ell=0$.
Again by Proposition \ref{p g=7 ambient Omega2},
it suffices to show $\ell+1\neq 2$, which follows from
$\ell+1=p+1\geq 4$.
\end{proof}

\section{Genus 9}

\begin{nota}\label{nota g=9}
Assume that $p> 2$.
Fix the following $6 \times 6$ matrix: 
\[
J:=
\begin{pmatrix}
O&E_3\\
-E_3&O
\end{pmatrix} \in \Mat_6(k).
\]
\begin{enumerate}
\item 
Set $G:=\Sp_6
:=
\left\{
g\in\SL_6
\mathrel{}\middle|\mathrel{}
g^tJg=J
\right\}$, which is a semisimple algebraic group. 
Set 
\[
T
:=
\left\{
\begin{pmatrix}
D&O\\
O&D^{-1}
\end{pmatrix}
\mathrel{}\middle|\mathrel{}
D=\diag(t_1,\ldots,t_3),\
t_i\in k^\times
\right\}, 
\]
\[
B
:=
\left\{
\begin{pmatrix}
A&AX\\
O&(A^{-1})^t
\end{pmatrix}
\mathrel{}\middle|\mathrel{}
A\in B_{\GL_3},\
X\in\Sym_3(k)
\right\}, 
\]
where $\Sym_3(k)
:=
\left\{
X\in\Mat_3(k)
\mathrel{}\middle|\mathrel{}
X^t=X
\right\}$ and 
$B_{\GL_3}$ denotes  the lower triangular Borel subgroup of $\GL_3$.
Then $T$ is a maximal torus and $B$ is a Borel subgroup of $\Sp_6$ satisfying $\Sp_6 \supset B \supset T$. 
\item 
Set 
 $W:=k^3\oplus 0\subset k^3\oplus k^3 =k^6$, 
 which is a Lagrangian subspace of $k^6$ 
 with respect to the non-degenerate alternating bilinear form
$(v,w) \mapsto v^tJw$ induced by $J$. 
The stabiliser of $W$ is the maximal parabolic subgroup
\[
P:=\Stab_{\Sp_6}(W)
=
\left\{
\begin{pmatrix}
A&AX\\
O&(A^{-1})^t
\end{pmatrix}
\mathrel{}\middle|\mathrel{}
A\in\GL_3,\ 
X\in\Sym_3(k)
\right\}.
\]
We have $P=L\ltimes R_u(P)$ for the  unipotent radical
\[
R_u(P)
=
\left\{
\begin{pmatrix}
I_3&X\\
O&I_3
\end{pmatrix}
\mathrel{}\middle|\mathrel{}
X\in\Sym_3(k)
\right\},
\]
and the Levi subgroup 
\[
L
:=
\left\{
\begin{pmatrix}
A&O\\
O&(A^{-1})^t
\end{pmatrix}
\mathrel{}\middle|\mathrel{}
A\in\GL_3
\right\}
\simeq\GL_3
\]
of $P$ containing $T$. Moreover, we have 
\[
B_L := B\cap L
=
\left\{
\begin{pmatrix}
A&O\\
O&(A^{-1})^t
\end{pmatrix}
\mathrel{}\middle|\mathrel{}
A\in B_{\GL_3}
\right\}
\simeq B_{\GL_3}.
\]
For $F:= G/B$ and $\Sigma := G/P$, let 
$\pi:F=G/B\to \Sigma = G/P$ be the natural projection. 
\item 
For $1\leq i\leq3$, let $\epsilon_i\in X(T)$ be the character
defined by
\[
\epsilon_i\bigl(\diag
(t_1,\ldots,t_3,t_1^{-1},\ldots,t_3^{-1})\bigr)
:=t_i.
\]
Set $\nu :=(3, 1, 0) := 3\epsilon_1 + \epsilon_2 \in X(T)$ and let $\lambda \in X(B)$ be the unique extension of $\nu$. 
\end{enumerate}
\end{nota}

\begin{rem}\label{r g=9 fund wt}
We use Notation \ref{nota g=9}. 
Set
\[
e_1:=-\epsilon_3,\qquad
e_2:=-\epsilon_2,\qquad
e_3:=-\epsilon_1.
\]
Then
\[
\alpha_1=e_1-e_2,\qquad
\alpha_2=e_2-e_3,\qquad
\alpha_3=2e_3
\]
form a set of simple roots for which $B$ is the negative Borel
subgroup and $P$ is the reduced maximal parabolic subgroup
corresponding to $\alpha_3$.
The corresponding fundamental weights are
\[
\omega_1=e_1,\qquad
\omega_2=e_1+e_2,\qquad
\omega_3=e_1+e_2+e_3.
\]
We have 
\[
\nu =(3, 1, 0) = 3\epsilon_1 + \epsilon_2 = -3e_3-e_2 = \omega_1 +2\omega_2 -3\omega_3. 
\]
\end{rem}

\begin{rem}
We use Notation \ref{nota g=9}. 
Then it is well known that 
$\Sigma$ is isomorphic to the Lagrangian Grassmannian variety
$\LG(3,6)$
parametrising the $3$-dimensional Lagrangian subspaces of $k^6$. 
For the universal subbundle $\cU$ on $\Sigma$, 
it is well known that 
\[
\dim \Sigma=6,\qquad
\det\mathcal U\simeq\mathcal O_\Sigma(-1),\qquad
%\Omega_\Sigma^1 \overset{(\star)}{\simeq}\Sym^2\mathcal U,\qquad
\omega_\Sigma^{-1} \simeq \MO_{\Sigma}(4). 
\]
Moreover, $\pi^*\MO_{\Sigma}(1) \simeq \cL_{\omega_3}$. 
%where $(\star)$ follows from Lemma \ref{l g=9 Omega1}. 
For the closed embedding $\Sigma \subset \P^{13}$ induced by $|\MO_{\Sigma}(1)|$, 
an arbitrary prime Fano threefold $X$ of genus $9$ is 
isomorphic to $\Sigma \cap \P^{10}$ for some $10$-dimensional linear subvariety $\P^{10} \subset \P^{13}$ \cite[Theorem 1.2]{KTprime}. 
\end{rem}

\begin{lem}\label{l genus 9 Schur 31}
Assume that $p>2$.
Let $B_{\GL_3}\subset\mathrm{GL}_3$ be the lower triangular Borel subgroup and let $U$ be the standard $3$-dimensional $\GL_3$-module.
Then the following hold: 
\begin{enumerate}
\item 
$H^0(\mathrm{GL}_3/B_{\GL_3},\cL_{(3,1,0)})$ 
is a simple $\mathrm{GL}_3$-module. 
\item There is a $\GL_3$-module isomorphism 
\[
\wedge^2(\Sym^2U)
\simeq
H^0(\mathrm{GL}_3/B_{\GL_3},\cL_{(3,1,0)}). 
\]
\end{enumerate}
\end{lem}

\begin{proof}
Set $\nu:=(3,1,0)$.

Let us show (1). 
Suppose that there exists a simple composition factor $L(\mu)$ of 
$H^0(\mathrm{GL}_3/B_{\GL_3},\cL_\nu)$ of highest weight 
$\mu=(\mu_1,\mu_2,\mu_3)\in\Z^3$ with $\mu\neq\nu$. 
Since the weight $\nu$ occurs with multiplicity one, it is enough to derive a contradiction. 
Since the $\mathrm{GL}_3$-module
$H^0(\mathrm{GL}_3/B_{\GL_3},\cL_\nu)$ is a homogeneous polynomial representation of degree $4$, 
we get 
\begin{enumerate}
\item[(i)] $\sum_{i=1}^3\mu_i=4$ and $\mu_1\geq\mu_2\geq\mu_3\geq0$. 
\end{enumerate}
As $\nu$ is dominant, so is $\nu+\rho$, and hence $\nu^+=\nu$. 
We have $\rho=(1,0,-1)$ by direct computation. 
By the strong linkage principle (Theorem \ref{t SLP}), 
the following hold: 
\begin{enumerate}
\item[(ii)] $\mu\leq\nu$, and 
\item[(iii)] 
$\nu+\rho=w(\mu+\rho)+p\alpha$ 
for some $w\in S_3$ and $\alpha\in\Z\Phi\subset\Z^3$. 
\end{enumerate}
As $\mu\neq\nu$, (i) and (ii) imply 
$\mu\in\left\{(2,2,0),(2,1,1)\right\}.$ 
We have
$\nu+\rho=(4,1,-1)$. 

If $\mu=(2,2,0)$, then $\mu+\rho=(3,2,-1),$ 
and hence (iii) leads to the following contradiction:
\[
p\Z\ni
(4^2+1^2+(-1)^2)
-
(3^2+2^2+(-1)^2)
=
4.
\]
If $\mu=(2,1,1)$, then
$\mu+\rho=(3,1,0),$ 
and hence (iii) leads to the following contradiction:
\[
p\Z\ni
(4^2+1^2+(-1)^2)
-
(3^2+1^2+0^2)
=
8.
\]
Thus (1) holds.

Let us show (2). We have 
\[
\dim\wedge^2(\Sym^2U)
=
\binom{\dim\Sym^2U}{2}
=
\binom{6}{2}
=
15.
\]
As $\nu$ is dominant, 
the Weyl dimension formula \cite[Corollary II.5.11]{Jan03} implies 
\[
\dim H^0(\mathrm{GL}_3/B_{\GL_3},\cL_\nu)
=
\prod_{1\leq i<j\leq3}
\frac{\nu_i-\nu_j+j-i}{j-i}
=
15.
\]
By (1), it is enough to show that $\nu$ is the largest weight of 
$\wedge^2(\Sym^2U)$. 
Let $U=ku_1 \oplus ku_2 \oplus ku_3$ be the weight decomposition, 
where $u_1, u_2,u_3$ are of weight
$(1,0,0)$, $(0, 1, 0)$, $(0,0,1)$. 
By using the weight decomposition
\[
\wedge^2(\Sym^2U)
=
\bigoplus_{(a,b,c,d)\in I}
k(u_au_b)\wedge(u_cu_d)
\]
for a suitable subset $I\subset\Z^4$, 
we see that
$u_1^2\wedge(u_1u_2)\in\wedge^2(\Sym^2U)$ 
is of weight $\nu=(3,1,0)$ and $\nu$ is the largest weight of $\wedge^2(\Sym^2U)$. 
Thus (2) holds. 
\end{proof}

\begin{lem}\label{l g=9 Omega1}
We use Notation \ref{nota g=9}. 
%Assume $p\neq 2$. 
%Set $\Sigma := \LG(3,6)$ and 
Let $\cU$ be the universal subbundle on $\Sigma$. 
Then $\Omega_\Sigma^1\simeq\Sym^2\cU$. 
\end{lem}

%Cheong–Choe–Hitching, Irreducibility of Lagrangian Quot schemes over an algebraic curve, Math. Z. 300 (2022), 1265–1289, proof of Proposition 2.4(b)
%\cite[p.~1268, proof of Proposition 2.4(b)]{CCH22}
\begin{proof}
Let $\cU_{\Gr(3, 6)}$ (resp. $\cQ_{\Gr(3, 6)}$) 
be the universal subbundle (resp.\ quotient bundle) on $\Gr(3, 6)$. 
Set $\cU := \cU_{\Gr(3, 6)}|_{\Sigma}$ and $\cQ := \cQ_{\Gr(3, 6)}|_\Sigma$. 
By definition, there exists an element 
$s\in
H^0(
\Gr(3, 6),
\wedge^2\cU_{\Gr(3, 6)}^\vee
)$ whose zero locus scheme is $\Sigma$.
It holds that 
\[
\codim_{\Gr(3, 6)}\,\Sigma
= 9-6 = 3
=
\rank(
\wedge^2\cU_{\Gr(3, 6)}^\vee
). 
\]
Thus the section $s$ is regular, and hence we have the conormal exact sequence
\[
0
\to 
\wedge^2\cU
\to 
\Omega_{\Gr(3, 6)}^1|_\Sigma
\to 
\Omega_\Sigma^1
\to 
0.
\]

As
$\Omega_{\Gr(3, 6)}^1
\simeq
\cU_{\Gr(3, 6)}\otimes\cQ_{\Gr(3, 6)}^\vee$, we have $\Omega_{\Gr(3, 6)}^1|_\Sigma
\simeq
\cU\otimes\cQ^\vee$. 
For $o := [W]  = eP \in  \Sp_6/P = \Sigma$, we have the following $P$-module isomorphism: 
\[
\cQ_o = k^6/W \overset{(\star)}{\simeq} W^\vee = \cU^\vee_o, 
\]
where $(\star)$ is induced by 
the symplectic form $\omega : k^6 \times k^6 \to k$.  
This $P$-module isomorphism corresponds to 
a $\Sp_6$-equivariant isomorphism $\cQ \simeq \cU^\vee$. 
Therefore,
% \[
% \cQ_{\Gr(3, 6)}^\vee|_\Sigma\simeq\cU
% \]
% and hence
$\Omega_{\Gr(3, 6)}^1|_\Sigma
\simeq
\cU\otimes\cU.$ 
For $U := \cU_o$, we now finish the proof by assuming that 
\begin{enumerate}
\item[($\star\star$)] $\Sym^2 U$ is a simple $P$-module. 
\end{enumerate}
%Here $U := \cU_o$. 
We have  the $\Sp_6$-equivariant exact sequence $0 \to \wedge^2 \cU \to \cU \otimes \cU \to \Omega^1_\Sigma \to 0$ 
and the decomposition $\cU\otimes\cU
\simeq
\Sym^2\cU
\oplus
\wedge^2\cU$ ensured by $p>2$. 
Comparing these, we see that 
the semisimplifications of $\Sym^2 U (=\Sym^2 \cU_o)$ and $\Omega^1_{\Sigma, o}$ as $P$-modules coincide. 
This, together with  ($\star\star$), induces a $P$-module isomorphism $\Sym^2 U \simeq \Omega^1_{\Sigma, o}$, which corresponds to $\Sym^2 \cU \simeq \Omega^1_{\Sigma}$.

It suffices to prove ($\star\star$). 
%By ..., it is enough to show that $\Sym U$ is a simple $\GL_3$-module. 
As $\cL_{(2, 0, 0)} \simeq q^* \MO_{\P^2}(2)$ for some contraction $q: \GL_3/B_{\GL_3} \to \P^2$, 
it is easy to see that 
$\Sym^2U
\simeq
H^0(\GL_3/B_{\GL_3},\cL_{(2,0,0)})$. 
Recall that $H^0(\GL_3/B_{\GL_3},\cL_{(2,0,0)})$ 
has simple socle \cite[Proposition II.2.3]{Jan03} and 
$\Sym^2U$ is completely
reducible for $p>2$ \cite[Corollary (2.6e)]{Gre07}. 
Thus $\Sym^2U$ is a simple $\GL_3$-module. 
Since $L \simeq \GL_3$ is a subgroup of $P$, ($\star\star$) holds. 
\qedhere

\end{proof}

\begin{lem}\label{l genus 9 Omega2 flag v4}
We use Notation \ref{nota g=9}. 
Then
\[
R\pi_*\cL_{\omega_1+2\omega_2-3\omega_3}
\simeq
\pi_*\cL_{\omega_1+2\omega_2-3\omega_3}
\simeq
\Omega_\Sigma^2.
\]
\end{lem}

\begin{proof}
%We use Notation \ref{nota g=9}. 
By Remark \ref{r g=9 fund wt}, 
we have $\lambda=\omega_1+2\omega_2-3\omega_3$. 
%By Remark \ref{r g=9 fund wt}, the restriction of $\lambda$ to$B_L$ is  $\nu=(3,1,0)$. 
Set $o:=eP\in G/P = \Sigma$. 
Note that $\cL_\lambda$ is $\pi$-nef, because 
$\nu =(3, 1, 0)$ is dominant and 
$\cL_\lambda|_{\pi^{-1}(o)} = \cL_\lambda|_{P/B} =\cL_\lambda|_{L/B_L} =\cL_\nu$. 
Therefore, $R\pi_*\cL_\lambda
\simeq
\pi_*\cL_\lambda$. 
We have the following $L$-module
isomorphisms:
\[
\pi_*\cL_\lambda|_o
\simeq
H^0(P/B,\cL_\lambda|_{P/B})
\simeq
H^0(L/B_L,\cL_\nu).
\]
On the other hand, Lemma \ref{l g=9 Omega1} gives
$\Omega_{\Sigma,o}^2
\simeq
\wedge^2(\Sym^2\cU_o).$ 
By Lemma \ref{l genus 9 Schur 31}, we obtain an $L$-module
isomorphism $\wedge^2(\Sym^2\cU_o)
\simeq
H^0(L/B_L,\cL_\nu)$, and this $L$-module is simple. 
Hence we get an $L$-module isomorphism
\[
\theta:
\pi_*\cL_\lambda|_o
\xrightarrow{\simeq}
\Omega_{\Sigma,o}^2
\]
between the $P$-modules.
Since both sides are simple $L$-modules, they are simple
$P$-modules. 
Therefore, the actions of $R_u(P)$ on both sides
are trivial \cite[Lemma 19.16]{Mil17}. 
By $P=L\ltimes R_u(P)$, $\theta$ is a $P$-module isomorphism, 
which corresponds to $\pi_*\cL_\lambda\xrightarrow{\simeq} \Omega_\Sigma^2$. 
\end{proof}

\begin{prop}\label{p genus 9 Omega2 vanishing v4}
Assume that $p>2$. Set $\Sigma := \LG(3, 6)$. Then 
$H^i(\Sigma, \Omega_\Sigma^2(m))=0$ if one of the following holds: 
\begin{enumerate}
\item $m\geq 2$ and $i\geq 1$. 
\item $m=1$ and $i \neq 1$. 
\item $m=0$ and $i \neq 2$. 
\end{enumerate}
\end{prop}

\begin{proof}
The case (3) is settled by Remark \ref{r non-diag vanishing}. 
We use Notation \ref{nota g=9}. 
Lemma \ref{l genus 9 Omega2 flag v4} implies  
\[
H^i(\Sigma, \Omega_\Sigma^2(m))
\simeq
H^i(
F,
\cL_{\omega_1+2\omega_2+(m-3)\omega_3}
)
\]
for every $i$. 
We have  $\cL_{\omega_1+2\omega_2+(m-3)\omega_3}
\otimes\MO_F(-K_F)
\simeq
\cL_{3\omega_1+4\omega_2+(m-1)\omega_3}$. 
If (1) (resp.\ (2)) holds, 
then the assertion follows from 
Proposition \ref{p flag Kodaira} (resp.\ 
Proposition \ref{p almost ample}). 
\qedhere

\end{proof}

\begin{thm}\label{t genus 9 F-split v4}
Assume $p>2$. 
Let $X$ be a prime Fano threefold of genus $9$. 
%over an algebraicallyclosed field of characteristic $p>2$.
Then $X$ is $F$-split.
\end{thm}

\begin{proof}
It is enough to show $H^1(X, \Omega_X^2(p))=0$ 
(Proposition \ref{p quasi-F-split quadrics}). 
Since $X \subset  \Sigma:=\LG(3, 6)$  is a codimension-$3$ linear section \cite[Theorem 1.2]{KTprime}, 
the conormal sequence of $X\subset \Sigma$ can be written as follows: 
\[
0
\to 
\cO_X(-1)^{\oplus3}
\to 
\Omega_\Sigma^1|_X
\to 
\Omega_X^1
\to 
0.
\]
Applying $\wedge^2$, 
we obtain the following exact sequences for some vector bundle $\cF$ on $X$: 
\[
0 \to \cF\to 
\Omega_\Sigma^2|_X
\to 
\Omega_X^2
\to 
0, \quad 
0
\to 
\cO_X(-2)^{\oplus3} \to \cF 
\to 
\Omega_X^1(-1)^{\oplus 3}
\to 0.
\]
By $H^2(X, \MO_X(p-2))= H^2(X, \Omega_X^1(p-1))=0$ 
(Proposition \ref{p quasi-F-split quadrics}), 
we get $H^2(X, \cF(p))=0$. 
Thus the problem is reduced to proving $H^1(X, \Omega^2_\Sigma(p)|_X)=0$. 
The Koszul exact sequence of $X\subset \Sigma$ is given by 
\[
0
\to 
\cO_\Sigma(-3)
\to 
\cO_\Sigma(-2)^{\oplus3}
\to 
\cO_\Sigma(-1)^{\oplus3}
\to 
\cO_\Sigma
\to 
\cO_X
\to 
0.
\]
Thus it is enough to prove
\begin{enumerate}
\item[(i)] $H^1(
\Sigma,
\Omega_\Sigma^2(p)
)
= H^2(
\Sigma,
\Omega_\Sigma^2(p-1)
)
=0$ 
\item[(ii)] 
$H^3(
\Sigma,
\Omega_\Sigma^2(p-2)
)
=
H^4(
\Sigma,
\Omega_\Sigma^2(p-3)
)
=0.$ 
\end{enumerate}
If $p \geq 5$, then (i) and (ii) follow from 
Proposition \ref{p genus 9 Omega2 vanishing v4}(1). 
Assume $p=3$. 
Then (i) (resp. (ii)) holds by 
Proposition \ref{p genus 9 Omega2 vanishing v4}(1) 
(resp.\ 
Proposition \ref{p genus 9 Omega2 vanishing v4}(2)(3)). 
\end{proof}

\begin{prop}
Assume $p=2$. 
Then there exists a prime Fano threefold $X$ of genus $9$ which is not $F$-split. 
\end{prop}

\begin{proof}
Let $\LG(3, 6) \subset \P^{13}$ be the closed embedding  induced by the 
ample generator. 
By \cite[Propositions 2.37 and 2.38]{Mac26}, 
there exists 
a smooth hyperplane section $H$ of $\LG(3,6)\subset \P^{13}$ 
such that $H \simeq G_2/P$ 
and $H\not\simeq G_2/P_{\red}$ 
for some non-reduced parabolic subgroup $P$ of $G_2$. 
Then $H$ is not $F$-split by \cite{LM97} (cf.\ \cite[Theorem 6.4.4]{AWZ21}). 
Take two general hyperplane sections $H', H'' \in |\MO_{\LG(3, 6)}(1)|$. 
By the Bertini theorem and adjunction formula, 
we see that $X := H \cap H' \cap H''$ is a prime Fano threefold of genus $9$.

Suppose that $X$ is $F$-split. 
It suffices to derive a contradiction. 
We have 
\[
H \supset H \cap H' =:Y \supset H \cap H' \cap H'' = X. 
\]
It follows from \cite[Lemma 2.7]{CTW17} that $(Y, X)$ is $F$-split, and hence $Y$ is $F$-split. 
Applying the same argument again, we see that $(H, Y)$ is $F$-split, and hence $H$ is $F$-split, which is absurd. 
\end{proof}

\section{Genus 10}

\begin{nota}\label{n g=10}
Let $G$ be a simply connected semisimple algebraic group of type $G_2$. 
Fix 
\[
T \subset B \subset P \subset G,
\]
where $T$ is a maximal torus, $B$ is the negative Borel subgroup, and 
$P$ is  the maximal 
reduced parabolic subgroup corresponding to the long simple root. 
Let $L \subset P$ be the Levi subgroup containing $T$. 
%Fix a Borel subgroup $B$ contained in $P$. 
For $F := G/B$ and $\Sigma:=G/P$, 
let 
\[
\pi : F = G/B \to \Sigma= G/P
\]
be the natural morphism. 
For the tautological subbundle $\cU$ on $\Sigma$, 
we have $\wedge^2 \cU \simeq \MO_\Sigma(-1)$, $F = 
G/B \simeq \P_\Sigma(\cU)$, and 
$\pi$ coincides with the induced $\P^1$-bundle 
$\P_\Sigma(\cU) \to \Sigma$  \cite[Lemma 8.3]{Kuz06}. 
Let $\xi$ be the tautological divisor on $F$ of $\pi : 
F \simeq \P_\Sigma(\cU) \to \Sigma$. 
Note that $\Sigma$ is a $5$-dimensional flag variety 
such that $\omega_\Sigma \simeq \MO_\Sigma(-3)$. 
%Let $\Sigma = G / P \subset \P^{13}$ be the closed embedding nduced by $|\MO_\Sigma(1)|$. 
\end{nota}

\begin{rem}\label{r genus 10 tautological GL2}
We use Notation \ref{n g=10}. 
Then $L \simeq \GL_2$ \cite[Proposition A.2.1]{And09}. 
For $o:=eP\in \Sigma$, 
$U:=\cU_o$ is the standard  $2$-dimensional representation of
$L \simeq \GL_2$ \cite[Corollary A.6.2]{And09}. 
\end{rem}

\begin{prop}\label{p genus 10 cotangent sequence}
We use Notation \ref{n g=10}. 
Assume $p>3$. 
Then there is an exact sequence
\[
0
\to
\cO_\Sigma(-1)
\to
\Omega_\Sigma^1
\to
(\Sym^3\cU)(1)
\to
0.
\]
\end{prop}

\begin{proof}
Set $\g:=\Lie(G)$ and $\p := \Lie(P)$. 
Let $\alpha_1$ and $\alpha_2$ be the short and long simple roots of $G$,
respectively.
The positive roots of $G$ are
\[
\alpha_1,
\qquad
\alpha_2,
\qquad
\alpha_1+\alpha_2,
\qquad
2\alpha_1+\alpha_2,
\qquad
3\alpha_1+\alpha_2,
\qquad
3\alpha_1+2\alpha_2.
\]
We have the following graded Lie algebra structure:
\[
\mathfrak g
= 
\mathfrak t \oplus \bigoplus_{\alpha} \mathfrak g_{\alpha}=
\mathfrak g_{-2}
\oplus
\mathfrak g_{-1}
\oplus
\mathfrak g_0
\oplus
\mathfrak g_1
\oplus
\mathfrak g_2, 
\]
where $\g_d$ is the direct sum of the root spaces 
$\g_{\beta}$ such that the coefficient of $\alpha_2$ in 
$\beta$ is equal to $d$. 
More explicitly, we have
\[
\begin{aligned}
\g_{-2}
&=
\g_{-(3\alpha_1+2\alpha_2)},\\
\g_{-1}
&=
\g_{-\alpha_2}
\oplus
\g_{-(\alpha_1+\alpha_2)}
\oplus
\g_{-(2\alpha_1+\alpha_2)}
\oplus
\g_{-(3\alpha_1+\alpha_2)},\\
\g_0
&=
\mathfrak t
\oplus
\g_{\alpha_1}
\oplus
\g_{-\alpha_1},\\
\g_1
&=
\g_{\alpha_2}
\oplus
\g_{\alpha_1+\alpha_2}
\oplus
\g_{2\alpha_1+\alpha_2}
\oplus
\g_{3\alpha_1+\alpha_2},\\
\g_2
&=
\g_{3\alpha_1+2\alpha_2}.
\end{aligned}
\]
% For example, $\mathfrak g_0 = \mathfrak t \oplus 
% \mathfrak g_{\alpha_1} \oplus \mathfrak g_{-\alpha_1}$ 
% and $\g_2 = \mathfrak g_{3\alpha_1+2\alpha_2}$. 
As $\dim \mathfrak t =2$, we obtain 
\[ 
\dim \mathfrak g_{-2}=1,\qquad \dim \mathfrak g_{-1}=4,\qquad \dim \mathfrak g_0=4,\qquad \dim \mathfrak g_1=4,\qquad \dim \mathfrak g_2=1. 
\]
Recall that %an algebraic group $H$ acts on $\Lie(H)$, and hence 
$\g$ is a $G$-module, and hence a $P$-module. 
Moreover, each $\mathfrak g_{\leq d}:=\bigoplus_{e \leq d} \g_e$ is a $P$-submodule of $\mathfrak g$. 
%In particular, $\p  = \g_{\leq 0}$ a $P$-submodule. 
% In particular, $\g_{\leq 1}/\g_{\leq 0}$ in $\mathfrak g/\g_{\leq 0}$ is a $P$-submodule. 
Hence we get a natural  exact sequence of $P$-modules: 
% Thus the tangent representation at the base point $o:=eP$ has a
% $P$-stable filtration
\[
0
\to
\g_{\leq 1}/\g_{\leq 0}
\to
\g/\g_{\leq 0}
\to
\g/\g_{\leq 1}
\to
0, 
\]
which induces 
the corresponding exact sequence 
\begin{equation}\label{e1 genus 10 cotangent sequence}
0 \to \cE(\g_{\leq 1}/\g_{\leq 0}) \to \cE(\g/\g_{\leq 0}) 
\to \cE(\g/\g_{\leq 1}) \to 0 
\end{equation}
of homogeneous vector bundles on $G/P =\Sigma$, 
where $\cE(-) := G \times^{P} (-)$. 
Set 
$\g'_{1} := \g_{\leq 1}/ \g_{\leq 0},$ which is a $P$-module. 
%Let $L$ be the Levi subgroup $L$ of $P$ containing $T$. 
Furthermore, $\p = \Lie(P) =\g_{\leq 0}=
\mathfrak g_0
\oplus
\mathfrak g_{-1}
\oplus
\mathfrak g_{-2}.$ 

\setcounter{step}{0}

\begin{step}\label{s1 genus 10 cotangent sequence}
$\cE(\g/\g_{\leq 0}) \simeq T_\Sigma$ and $\cE(\g/\g_{\leq 1}) \simeq \MO_\Sigma(1)$. 
\end{step}

\begin{proof}[Proof of Step \ref{s1 genus 10 cotangent sequence}]
It holds that $\cE(\g/\g_{\leq 0}) 
=\cE(\g/\p) \simeq 
T_{G/P} = T_\Sigma$. 
Let us prove that 
$\cE(\g/\g_{\leq 1}) \simeq \MO_\Sigma(1)$. 
Recall that $\MO_\Sigma(1) \simeq \cL_{\omega_2}$. 
As $\dim (\g/\g_{\leq 1}) = \dim (\g_{\leq 2}/\g_{\leq 1}) = \dim \g_2 =1$, 
it is enough to establish a $T$-module isomorphism  
$\g/\g_{\leq 1} \simeq \g_{\omega_2}$, 
which is ensured by 
%which follows from %We have a $T$-module isomorphism 
$\g/\g_{\leq 1} \simeq 
\g_2 = \g_{3\alpha_1+2\alpha_2} = \g_{\omega_2}$. 
This completes the proof of Step \ref{s1 genus 10 cotangent sequence}.
\end{proof}

Set %$\mathfrak t:=\Lie(T)$, 
$o:=eP\in\Sigma$ and $U:=\cU_o$. 

\begin{step}\label{s2 genus 10 cotangent sequence}
There is an $L$-module isomorphism 
\[
\g'_{1} \otimes_k k_{\chi} \simeq \Sym^3(U)
\]
between $P$-modules for some $\chi \in X(P)$. 
\end{step}

\begin{proof}[Proof of Step \ref{s2 genus 10 cotangent sequence}]
By $S := [L, L]  \triangleleft L$, $X(L) = X(P)$, and 
Lemma \ref{l character twist} below, it is enough to show that 
\begin{enumerate}
\item[(i)] $\Sym^3(U)$ is a simple $S$-module, and 
\item[(ii)] 
there is an $S$-module isomorphism $\g'_1 \simeq \Sym^3(U)$. 
\end{enumerate}
In what follows, we use the identification $L = \GL_2$ (Remark \ref{r genus 10 tautological GL2}). 
We then get $S = [L, L] = \SL_2$.

Let us show (i). 
Since $U := \cU_o$ is the standard $2$-dimensional representation of $L= \GL_2$ (Remark \ref{r genus 10 tautological GL2}), 
$\Sym^3 U$ is a simple $\GL_2$-module by $p>3$ \cite[Corollary (2.6e)]{Gre07}, and hence a simple $\SL_2$-module. 
Thus (i) holds. 
Let us show (ii). 
Note that $\dim(\Sym^3 U) =4$ and  $\Sym^3 U$ is of highest weight $3$. 
By (i), it is enough to prove that $\dim \g'_1 =4$ and the largest weight of $\SL_2$-module $\g'_1$ is $3$. 
These follow from the $T$-module isomorphism
\[
\g'_1 \simeq \g_1=
\g_{\alpha_2}
\oplus
\g_{\alpha_1+\alpha_2}
\oplus
\g_{2\alpha_1+\alpha_2}
\oplus
\g_{3\alpha_1+\alpha_2}, 
\]
together with $\langle \alpha_2,\alpha_1^\vee\rangle=-3,
\langle \alpha_1+\alpha_2,\alpha_1^\vee\rangle=-1, 
\langle 2\alpha_1+\alpha_2,\alpha_1^\vee\rangle=1, 
\langle 3\alpha_1+\alpha_2,\alpha_1^\vee\rangle=3.$ 
Thus (ii) holds. 
This completes the proof of Step \ref{s2 genus 10 cotangent sequence}.
\qedhere

\end{proof}

\begin{step}\label{s3 genus 10 cotangent sequence}
There is a $P$-module isomorphism 
\[
\g'_{1} \otimes_k k_{\chi} \simeq \Sym^3(U)
\]
between $P$-modules for some $\chi \in X(P)$. 
\end{step}

\begin{proof}[Proof of Step \ref{s3 genus 10 cotangent sequence}]
By Step \ref{s2 genus 10 cotangent sequence} and $P =L \cdot R_u(P)$, 
it is enough to show that the $R_u(P)$-actions on 
$\g'_{1} \otimes_k k_{\chi}$ and $\Sym^3(U)$ are trivial. 
The $R_u(P)$-action on $k_{\chi}$ is trivial, as $R_u(P)$ is unipotent. 
The $R_u(P)$-action on $\g'_{1}$ is also trivial by  \cite[Proposition 22.14]{Mil17}. 
To summarise, the $R_u(P)$-action on $\g'_{1} \otimes_k k_{\chi}$ is trivial. 

Concerning $\Sym^3(U)$, 
it is enough to prove that the $R_u(P)$-action on $U$ is trivial. 
Since $U$ is a simple $S$-module, $U$ is also a simple $P$-module.
Hence $R_u(P)$ acts trivially on $U$ by \cite[Lemma 19.16]{Mil17}. 
This completes the proof of Step \ref{s3 genus 10 cotangent sequence}.
\end{proof}

\begin{step}\label{s4 genus 10 cotangent sequence}
There is an exact sequence 
$0 \to \MO_\Sigma(-1) \to \Omega_\Sigma^1 \to (\Sym^3 \cU)(1) \to 0$. 
\end{step}

\begin{proof}[Proof of Step \ref{s4 genus 10 cotangent sequence}]
By the $P$-module isomorphism $\g'_{1} \otimes_k k_{\chi} \simeq \Sym^3(U)$ established in Step \ref{s3 genus 10 cotangent sequence}, we get  
\[
\cE(\g'_{1}) \otimes_{\MO_\Sigma} \cL_{\chi} 
\simeq 
\cE(\g'_{1} \otimes_k k_{\chi}) \simeq \cE(\Sym^3(U)) 
\simeq \Sym^3( \cE(U)) \simeq \Sym^3 \cU.
\]
As $p >3$, we have 
$(\Sym^3 \cU)^{\vee} \simeq \Sym^3 (\cU^{\vee}) 
\simeq \Sym^3(\cU(1))$. 
These, together with (\ref{e1 genus 10 cotangent sequence}) 
and Step \ref{s1 genus 10 cotangent sequence}, 
induce an exact sequence %which implies %To summarise, we get 
\[
0 \to \MO_\Sigma(-1) \to \Omega_\Sigma^1 \to (\Sym^3 \cU)(n) \to 0
\]
for some $n \in \Z$. 
By $\rank \cU=2$ and $\rank (\Sym^3 \cU)=4$, we have 
\[
\det ((\Sym^3 \cU)(n)) \simeq \det(\Sym^3 \cU) 
\otimes \MO_\Sigma(4n) \simeq (\det \cU)^{\otimes 6} 
\otimes \MO_\Sigma(4n) \simeq \MO_\Sigma(4n-6).
\]
Taking the determinant bundles of the above exact sequence, 
we get $\MO_\Sigma(-3) \simeq \det \Omega_\Sigma^1 \simeq 
\MO_\Sigma(-1) \otimes \det ((\Sym^3 \cU)(n)) \simeq \MO_\Sigma(4n-7)$, 
which implies $n=1$. 
This completes the proof of Step \ref{s4 genus 10 cotangent sequence}.
\end{proof}
Step \ref{s4 genus 10 cotangent sequence} completes the proof of Proposition \ref{p genus 10 cotangent sequence}. 
\end{proof}

\begin{lem}\label{l character twist}
Let $L$ be an algebraic group and let $S\triangleleft L$ be a closed normal subgroup.
Let $V$ and $W$ be finite-dimensional $L$-modules. 
Assume that $V$ and $W$ are simple $S$-modules and isomorphic as $S$-modules. 
Then there exists 
an $L$-module isomorphism $V\otimes_k k_\chi\simeq W$ 
for some character $\chi\in X(L)$. 
\end{lem}

\begin{proof}
Since $S\triangleleft L$, the vector space $\Hom_S(V,W)$ is an $L$-submodule of $\Hom_k(V,W)$.
By Schur's lemma,
we have $\dim_k\Hom_S(V,W)=1.$ 
Hence $\Hom_S(V,W)\simeq k_\chi$ as $L$-modules for some $\chi\in X(L)$.
A nonzero element of $\Hom_S(V,W)$ then induces an $L$-module isomorphism $V\otimes_k k_\chi\simeq W$. 
\end{proof}

\begin{lem}\label{l wedge Sym3 rank2}
Let $\cE$ be a vector bundle of rank $2$ on a $\Z[1/6]$-scheme $X$.  
Then there is an exact sequence
\[
0
\to 
(\det\cE)^{\otimes3}
\to 
\wedge^2(\Sym^3\cE)
\xrightarrow{\varphi}
\det\cE \otimes 
\Sym^4\cE
\to 
0.
\]
\end{lem}

\begin{proof}
Consider the following $\MO_X$-module homomorphisms:   
\[
\delta:
\Sym^3\cE
\to 
\cE\otimes\Sym^2\cE, \qquad 
uvw \mapsto 
u\otimes vw
+
v\otimes uw
+
w\otimes uv
\]
and 
\[
\psi:
\wedge^2(\cE\otimes\Sym^2\cE)
\to
\wedge^2\cE\otimes\Sym^4\cE, \qquad 
(a\otimes q)\wedge(b\otimes r)
\mapsto 
(a\wedge b)\otimes qr.
\]
We set
$\varphi 
:=
\psi\circ(\wedge^2\delta)$, i.e., 
$
\varphi: 
\wedge^2(\Sym^3\cE)
\xrightarrow{\wedge^2 \delta}
\wedge^2(\cE\otimes\Sym^2\cE) \xrightarrow{\psi} 
\wedge^2\cE\otimes\Sym^4\cE.$

For every geometric fibre $V$ of $\cE$ of characteristic $p\geq 0$, 
the induced $\GL(V)$-module homomorphism  
$\varphi_V : \wedge^2(\Sym^3V)
\to 
\wedge^2V \otimes\Sym^4V$ is nonzero, because 
the following holds for a linear basis $x, y$ of $V$: 
\[
\varphi_V(x^3\wedge y^3)
=
\psi(
(\wedge^2\delta)(x^3\wedge y^3)
)
=
\psi(
\delta(x^3)\wedge\delta(y^3)
)
=
\psi(
(3x\otimes x^2)\wedge(3y\otimes y^2)
)
\]
\[
=
9\psi(
(x\otimes x^2)\wedge(y\otimes y^2)
)=
9(x\wedge y)\otimes x^2y^2
\neq
0.
\]
Since $\Sym^4 V$ is a simple $\GL(V)$-module by $p \not\in \{2, 3\}$, 
$\varphi_V$ is surjective. 
Hence $\varphi$ is surjective. 
Then  $\cK:=\Ker(\varphi)$ is a line bundle.
By $\det(\wedge^2(\Sym^3\cE))
\simeq
(\det\cE)^{\otimes18}$ and
$\det(\Sym^4\cE\otimes\det\cE)
\simeq
(\det\cE)^{\otimes15}$, we get 
$\cK \simeq (\det\cE)^{\otimes3}$, as required. 
\end{proof}

\begin{prop}\label{p genus 10 Omega2 vanishing v2}
We use Notation \ref{n g=10}. 
Assume that $p>3$.
Then $H^i(
\Sigma,
\Omega_\Sigma^2(m)
)
=0$ for every $i>0$ and every  $m\geq 2$.
\end{prop}

\begin{proof}
Fix an ample Cartier divisor $H$ on $\Sigma$ with $\MO_{\Sigma}(H) \simeq \MO_\Sigma(1)$. 
For every $a \geq 0$, we have $R\pi_*\MO_F(a\xi) = \pi_*\MO_F(a\xi) \simeq \Sym^a\cU$, 
which implies 
\[
H^i(
\Sigma,
(\Sym^a\cU)(m)
)
\simeq
H^i(
F,
\cO_F(a\xi+m\pi^*H)
).
\]
As $\rank \cU = 2$, we get 
$\cU(1)
\simeq
\cU^\vee$. 
In particular, $\cU(1)$ is globally generated, and hence 
$\xi+\pi^*H$ is nef and $\pi$-ample.
Applying $\wedge^2$ to %Taking the second exterior power of
an exact sequence $0
\to
\cO_\Sigma(-1)
\to
\Omega_\Sigma^1
\to
(\Sym^3\cU)(1)
\to
0$ (Proposition \ref{p genus 10 cotangent sequence}), 
we obtain
\[
0
\to
\Sym^3\cU
\to
\Omega_\Sigma^2
\to
(\wedge^2
(\Sym^3\cU)
)(2)
\to
0.
\]
As $p>3$, there is the following exact sequence 
(Lemma \ref{l wedge Sym3 rank2}): 
\[
0
\to
(\det\cU)^{\otimes 3}
\to
\wedge^2(\Sym^3\cU)
\to
\Sym^4\cU\otimes\det\cU
\to
0.
\]
Hence it suffices to prove that 
\[
H^{>0}(\Sigma, \MO_\Sigma( m-1)) =0 \qquad \text{and}
\]
\[
H^{>0}(F, \MO_F(3 \xi + m\pi^*H))= 
H^{>0}(F, \MO_F(4 \xi + (m+1)\pi^*H))= 0
\]
for every $m \geq 2$. % a\in \{3, 4\}$, and $n \geq 1$. 
The former one follows from the fact that $\MO_\Sigma( m-1)$ is nef. 
The latter one holds because 
the following divisor is ample for each $b  \in \{0, 1\}$ 
(Proposition \ref{p flag Kodaira}): 
\[
(3+b)\xi +(m+b)\pi^*H  -K_F 
\sim (3+b) (\xi +\pi^*H) + (m-1)\pi^*H + N, 
\]
where 
$N$ is a nef divisor on $F$ satisfying $-K_F \sim 2\pi^*H +N$. Here such a linear equivalence 
$-K_F \sim 2\pi^*H +N$ is ensured by the formula 
$\omega_{G/B}^{-1} \sim \cL_{2 \rho}$ for the sum $\rho$ of all the fundamental weights. 
\end{proof}

\begin{thm}\label{t genus 10 F-split v2}
Assume $p>3$. 
Let $X$ be a prime Fano threefold of genus $10$. 
Then the following hold: 
\begin{enumerate}
\item $H^1(X, \Omega_X^2(m))=0$ for every $m \geq 4$. 
\item $X$ is $F$-split. 
\end{enumerate}
\end{thm}

\begin{proof}
Since (1) implies (2) (Proposition \ref{p quasi-F-split quadrics}),  let us show (1). 
We use Notation \ref{n g=10}. 
We may assume that $X \subset \Sigma$ is a linear section of codimension two \cite[Theorem 1.2]{KTprime}. 
Fix $m \geq 4$. 
Applying $\wedge^2$ to the conormal exact sequence 
$0
\to
\MO_X(-1)^{\oplus 2}
\to
\Omega_\Sigma^1|_X
\to
\Omega_X^1
\to
0,$ we get the following exact sequences for some vector bundle $\cF$:
\[
0
\to
\cF
\to
\Omega_\Sigma^2|_X
\to
\Omega_X^2
\to
0, \quad 
0
\to
\MO_X(-2) 
\to
\cF
\to
\Omega_X^1(-1)^{\oplus 2}
\to
0.
\]
Since $H^{2}(X, \MO_X(m-2))= H^{2}(X, \Omega_X^1(m-1))=0$  (Theorem \ref{t Fano3 ANV}, Proposition \ref{p quasi-F-split quadrics}), 
it is enough to prove that $H^{>0}(X, \Omega^2_\Sigma(m)|_X)=0$. 
Since $X \subset \Sigma$ is a linear section of codimension two, 
the corresponding Koszul exact sequence is given as follows: 
%The Koszul resolution of the codimension-$2$ linear section$X\subset \Sigma$ is
\[
0
\to
\cO_\Sigma(-2)
\to
\cO_\Sigma(-1)^{\oplus2}
\to
\cO_\Sigma
\to
\cO_X
\to
0.
\]
Then the problem is reduced to showing 
$H^{>0}(\Sigma, \Omega^2_\Sigma(m-\ell)) =0$ for $\ell \in \{0, 1, 2\}$, 
which follows from 
$m -\ell \geq 2$ and Proposition \ref{p genus 10 Omega2 vanishing v2}. 
\end{proof}

\section{Genus 12}

\subsection{$H^{>0}(\Gr(r, n), 
(\wedge^a \cQ)^{\otimes k} \otimes 
 ((\wedge^{b} \cU)(1))^{\otimes \ell} \otimes \MO_{\Gr(r, n)}(t-1)) =0$}

\begin{nota}\label{n g=12 general}
For $1\leq r<n$, we set $s :=n-r$. 
We have the following cartesian diagram: 
\[
\begin{tikzcd}[column sep=-1.0em, row sep=1.5em]
& F:=\Fl(1,2,\ldots,n-1;n)
  \arrow[dl, "p_{\cU}"']
  \arrow[dr, "p_{\cQ}"]
  \arrow[dd, "\pi"]& \\
E_{\cU}:=\Fl(1,2,\ldots,r;n)
  \arrow[dr, "\pi_{\cU}"'] & &
E_{\cQ}:=\Fl(r,r+1,\ldots,n-1;n)
  \arrow[dl, "\pi_{\cQ}"] \\
& Y:=\Gr(r,n) &
\end{tikzcd}
\]
Let $\cU$ (resp.\ $\cQ$) be the universal subbundle
(resp.\ quotient bundle) on $Y$. 
\end{nota}

\begin{rem}\label{r g=12}
Set $G:=\GL_n$ and
$T:=
\left\{
\diag(t_1,\ldots,t_n)
\mathrel{}\middle|\mathrel{}
t_i\in k^\times
\right\}
\subset G$.
Let $B\subset G$ be the lower triangular Borel subgroup.
Let $B_{\GL_s}$  denote the lower
triangular Borel subgroup of $\GL_s$.
We have $G\supset P\supset Q\supset B\supset T$ for
\[
P:=
\left\{
\begin{pmatrix}
A&O\\
C&D
\end{pmatrix}
\mathrel{}\middle|\mathrel{}
A\in\GL_s,\ D\in\GL_r,\
C\in\Mat_{r\times s}(k)
\right\},
\]
\[
Q:=
\left\{
\begin{pmatrix}
A&O\\
C&D
\end{pmatrix}
\mathrel{}\middle|\mathrel{}
A\in B_{\GL_s},\ D\in\GL_r,\
C\in\Mat_{r\times s}(k)
\right\}
\]
Under the notation of Notation \ref{n g=12 general},
we have
$Y=G/P, E_{\cQ}=G/Q,$ and $F=G/B$. 
Moreover,
$P/Q\simeq\GL_s/B_{\GL_s}$. 
The surjective homomorphism of algebraic groups
\[
q:P\twoheadrightarrow\GL_s,\qquad
\begin{pmatrix}
A&O\\
C&D
\end{pmatrix}
\mapsto A
\]
induces $q|_Q:Q\twoheadrightarrow B_{\GL_s}$.
\end{rem}

\begin{prop}\label{p g=12 good filtQ}
We use Notation \ref{n g=12 general}.
Let $a_1, ..., a_k$ be non-negative integers.
Then there exist a sequence of locally free
$\MO_Y$-subsheaves
\[
(\wedge^{a_1} \cQ) \otimes 
\cdots \otimes (\wedge^{a_k} \cQ) =:\cF_0
\supset \cF_1 \supset \cdots \supset \cF_N =0
\]
and nef line bundles $\cL_0,\ldots,\cL_{N-1}$ on $E_{\cQ}$
such that
$\cF_i/\cF_{i+1}\simeq(\pi_{\cQ})_*\cL_i$
for every $i\in\{0,1,\ldots,N-1\}$.
\end{prop}

\begin{proof}
If $a_i>s$ for some $i$, then there is nothing to show. 
In what follows, we assume $a_i \leq s$ for every $i$ and use notation introduced in Remark \ref{r g=12}.
For $o:=eP\in G/P$,
the $G$-equivariant vector bundle $\cQ$ on $G/P$
corresponds to the $P$-module $\cQ_o$, where
$\dim_k\cQ_o=s$.
In particular, we have the induced surjective homomorphism
$q:P\to\GL_s$ of algebraic groups
(Remark \ref{r g=12}).
Note that the $P$-module structure on $\cQ_o$ is given by
this group homomorphism and the standard $\GL_s$-module
on $\cQ_o(=k^s)$.
For a dominant weight $\lambda$ of $\GL_s$ and
the lower triangular Borel subgroup $B_{\GL_s}$ of
$\GL_s$, we set
\[
\nabla(\lambda):=
H^0(\GL_s/B_{\GL_s},\cL_\lambda),
\]
which is a $\GL_s$-module.
Then $\wedge^a \cQ_o\simeq \nabla(1,..., 1, 0,\ldots,0)$, 
where the number of entries equal to $1$ is $a$. 

We say that a $\GL_s$-module $V$ has a
{\em good filtration} if there exists a sequence
$V=:F_0\supset F_1\supset\cdots\supset F_N=0$
of $\GL_s$-submodules such that
$F_i/F_{i+1}\simeq\nabla(\lambda_i)$ for every $i$,
where each $\lambda_i$ is a dominant weight of $\GL_s$.
In particular, $\wedge^{a_i} \cQ_o$ has a good filtration.
Tensor products of $\GL_s$-modules with good
filtrations again have good filtrations
\cite[Corollary 6.3]{Mat00}.
Consequently, the $\GL_s$-module
$(\wedge^{a_1} \cQ_o) \otimes 
\cdots \otimes (\wedge^{a_k} \cQ_o) $ corresponding to
$(\wedge^{a_1} \cQ) \otimes 
\cdots \otimes (\wedge^{a_k} \cQ)$ has a good filtration.
As $(\wedge^{a_1} \cQ_o) \otimes 
\cdots \otimes (\wedge^{a_k} \cQ_o) $ is a polynomial
$\GL_s$-module \cite[Subsection 2.6]{Gre07},
any weight $\mu=(\mu_1,\ldots,\mu_s)$ of
$(\wedge^{a_1} \cQ_o) \otimes 
\cdots \otimes (\wedge^{a_k} \cQ_o) $ satisfies
$\mu_i\geq0$ for every $i$ \cite[(3.2c)]{Gre07}.
Then this property also holds for the filtration quotients,
and hence their highest weights
$\lambda=(\lambda_1,\ldots,\lambda_s)$ satisfy
$\lambda_1\geq\cdots\geq\lambda_s\geq0$.

For the pullback
$\lambda'\in X(Q)$ of
$\lambda\in X(B_{\GL_s})$ by
$q|_Q:Q\to B_{\GL_s}$, we have the following
$P$-module isomorphisms:
\[
\bigl((\pi_{\cQ})_*\cL_{\lambda'}\bigr)_o
\simeq
H^0(P/Q,\cL_{\lambda'}|_{P/Q})
\simeq
H^0(\GL_s/B_{\GL_s},\cL_\lambda)
=
\nabla(\lambda),
\]
where $\nabla(\lambda)$ is regarded as a $P$-module
via the projection $P\to\GL_s$.
Corresponding to the good filtration
$(\wedge^{a_1} \cQ_o) \otimes 
\cdots \otimes (\wedge^{a_k} \cQ_o)=:F_0\supset F_1
\supset\cdots\supset F_N=0$,
we get a sequence of locally free $\MO_Y$-submodules
$(\wedge^{a_1} \cQ) \otimes 
\cdots \otimes (\wedge^{a_k} \cQ)=:\cF_0\supset\cF_1
\supset\cdots\supset\cF_N=0$
whose graded quotients are of the form
$(\pi_{\cQ})_*\cL_{\lambda'}$.

Thus it is enough to show that each $\cL_{\lambda'}$
is nef, which follows from
$\lambda_1\geq\cdots\geq\lambda_s\geq0$, 
$\cL_{\lambda'}
\simeq
\cL_{\sum_{i=1}^{s}\lambda_i\epsilon_i}$, and 
\[
\sum_{i=1}^{s}\lambda_i\epsilon_i =
\sum_{i=1}^{s-1}(\lambda_i-\lambda_{i+1})\omega_i
+\lambda_s\omega_s,
\]
where $\epsilon_i\in X(T)$ is defined by
$\epsilon_i(\diag(t_1,\ldots,t_n))=t_i$
and $\omega_i:=\epsilon_1+\cdots+\epsilon_i$
for every $i\in\{1,\ldots,s\}$.
Recall that the line bundle $\cL_{\omega_i}$ on
$E_{\cQ}$ is nef for every $i\in\{1,\ldots,s\}$.
\end{proof}

\begin{prop}\label{p g=12 good filtU}
We use Notation \ref{n g=12 general}.
Let $b_1, ..., b_{\ell}$ be non-negative integers.
Then there exist a sequence of locally free
$\MO_Y$-subsheaves
\[
( (\wedge^{b_1} \cU)(1)) \otimes 
\cdots \otimes ((\wedge^{b_\ell} \cU)(1)) =:\cF_0
\supset \cF_1 \supset \cdots \supset \cF_N =0
\]
and nef line bundles $\cL_0,\ldots,\cL_{N-1}$ on $E_{\cU}$
such that
$\cF_i/\cF_{i+1}\simeq(\pi_{\cU})_*\cL_i$
for every $i\in\{0,1,\ldots,N-1\}$.
\end{prop}

\begin{proof}
If $b_j>r$ for some $j$, then there is nothing to show. 
In what follows, we assume $b_j\leq r$ for every $j$.
Set $Y':=\Gr(s,(k^n)^\vee)$ and let $\cQ'$ be its universal quotient bundle. 
Then the canonical isomorphism
\[
\iota:Y\xrightarrow{\simeq}Y',\qquad
[W]\mapsto[W^\perp]
\]
satisfies $\iota^*\cQ'\simeq\cU^\vee$. Since
$\MO_Y(1)\simeq(\det\cU)^{-1}$, we have
\[
\iota^*(\wedge^{r-b_j}\cQ') \simeq 
\wedge^{r-b_j} \iota^*\cQ' \simeq \wedge^{r-b_j} \cU^\vee
\simeq (\wedge^{b_j}\cU)(1). 
\]
Note that the following canonical square diagram is the same as the one in Notation \ref{n g=12 general} up to isomorphisms: 
\[
\begin{tikzcd}[column sep=-1.0em, row sep=1.5em]
& \Fl(1,2,\ldots,n-1;(k^n)^\vee)
  \arrow[dl]
  \arrow[dr]
  \arrow[dd] & \\
\Fl(s,s+1,\ldots,n-1;(k^n)^\vee)
  \arrow[dr] & &
\Fl(1,2,\ldots,s;(k^n)^\vee)
  \arrow[dl] \\
& Y':=\Gr(s,(k^n)^\vee). &
\end{tikzcd}
\]
Then the assertion holds by applying Proposition \ref{p g=12 good filtQ} to
$\wedge^{r-b_1}\cQ' \otimes \cdots \otimes \wedge^{r-b_{\ell}} \cQ'$.
\end{proof}

\begin{prop}\label{p g=12 good filt}
We use Notation \ref{n g=12 general}.
Let $a_1, ..., a_k, b_1, ..., b_{\ell}$ be non-negative integers, 
where $k \geq 0$ and $\ell \geq 0$.
Then \[
H^i(Y, 
(\wedge^{a_1} \cQ) \otimes 
\cdots \otimes (\wedge^{a_k} \cQ) \otimes 
 (\wedge^{b_1} \cU)(1) \otimes 
\cdots \otimes (\wedge^{b_\ell} \cU)(1) \otimes \MO_Y(t-1)) =0
\]
for every $i>0$ and $t \geq 0$. 
\end{prop}

\begin{proof}
Set
\[
\cA:=
(\wedge^{a_1}\cQ)\otimes\cdots\otimes(\wedge^{a_k}\cQ),
\qquad
\cB:=
((\wedge^{b_1}\cU)(1))\otimes\cdots\otimes
((\wedge^{b_\ell}\cU)(1)).
\]
By Propositions \ref{p g=12 good filtQ} and
\ref{p g=12 good filtU}, the bundles $\cA$ and $\cB$
have filtrations whose graded quotients are respectively of
the form $(\pi_{\cQ})_*\cL$ and $(\pi_{\cU})_*\cM$,
where $\cL$ and $\cM$ are nef line bundles on
$E_{\cQ}$ and $E_{\cU}$.
Thus $\cA\otimes\cB$ has a filtration whose graded
quotients are of the form
$(\pi_{\cQ})_*\cL\otimes(\pi_{\cU})_*\cM$. 
Hence it is enough to prove 
\[
H^{>0}(Y, (\pi_{\cQ})_*\cL\otimes(\pi_{\cU})_*\cM (t-1))=0. 
\]

For the nef line bundle $\cN:=p_{\cQ}^*\cL\otimes p_{\cU}^*\cM$ on $F$, we have $R^{>0}\pi_*\cN =0$ and 
\[
\begin{aligned}
\pi_*\cN
&=
(\pi_{\cQ})_*(p_{\cQ})_*
(p_{\cQ}^*\cL\otimes p_{\cU}^*\cM
)\\
&\simeq 
(\pi_{\cQ})_*
(\cL\otimes (p_{\cQ})_*p_{\cU}^*\cM
)\\
&\simeq 
(\pi_{\cQ})_*
(\cL\otimes \pi_{\cQ}^*(\pi_{\cU})_*\cM
)\\
&\simeq(\pi_{\cQ})_*\cL
\otimes(\pi_{\cU})_*\cM.
\end{aligned}
\]
Therefore, we obtain 
\[
H^{>0}(Y, (\pi_{\cQ})_*\cL\otimes(\pi_{\cU})_*\cM (t-1)) 
\simeq H^{>0}(F, \cN \otimes \pi^*\MO_Y(t-1))\overset{(\star)}{=}0, 
\]
where $(\star)$ 
follows from Proposition \ref{p flag Kodaira}, 
which is applicable because 
$\cN \otimes \pi^*\MO_Y(t-1) \otimes \MO_F(-K_F)$ is ample. 
Here this ampleness is ensured by the formula 
$\MO_F(-K_F) \sim \cL_{2\rho}$.
\end{proof}

\subsection{F-splitting for genus 12}

\begin{nota}\label{n g=12}
Assume $p>2$. 
For $E:=\Fl(4,5,6;7)$ and 
$Y := \Gr(4, 7)$, 
let 
\[
\pi:E=\Fl(4,5,6;7)\to 
Y =\Gr(4, 7) 
\]
be the natural contraction. 
Let $\cU$ (resp.\ $\cQ$) be the universal subbundle (resp.\ quotient bundle) on $Y= \Gr(4, 7)$. 
\end{nota}

\begin{lem}\label{l genus 12 v6 polynomial vanishing}
We use Notation \ref{n g=12}. 
Let $\nu$ be a non-negative integer. 
Then 
\[
H^i(Y,\cQ^{\otimes \nu}(t))=0
\]
if one of the following holds:
\begin{enumerate}
\item $t \geq -4$ and $i>0$. 
\item $t=-5$ and $i >4$. 
\end{enumerate}
\end{lem}

\begin{proof}
By Proposition \ref{p g=12 good filtQ}, 
it is enough to prove the assertion after replacing $\cQ^{\otimes \nu}$ by $\pi_*\cL$ for a nef line bundle $\cL$ on $E=\Fl(4,5,6;7)$. 
As $\cL$ is $\pi$-nef, we get $R^{>0}\pi_*\mathcal L=0$. 
Thus the problem is  reduced to 
%the corresponding statement  for
\[
H^i(E,\mathcal L \otimes
\pi^*\mathcal O_Y(t))=0, 
\]
which follows from 
Proposition \ref{p flag Kodaira} and 
Proposition \ref{p almost ample general} because 
$-K_E \sim 2\omega_1 + 2\omega_2 + 5\omega_3$ and 
every fibre of the projection 
$E=\Fl(4,5,6;7) \to \Fl(5,6;7)$ is 
isomorphic to $\Gr(4, 5) \simeq \P^4$.
\end{proof}

\begin{lem}\label{l genus 12 v6 universal kernels}
We use Notation \ref{n g=12}. 
Let $\nu$ be a non-negative integer. 
Then the following hold: 
\begin{enumerate}
\item 
$H^i(Y,\mathcal U\otimes\cQ^{\otimes \nu}(t))=0$
if $t\ge-4$ and $i \geq 2$. 
\item $H^i(Y,\mathcal U\otimes\cQ^{\otimes \nu}(-5))=0$ 
for $i\geq 6$. 
\item 
$H^i(Y,\wedge^2\mathcal U\otimes\cQ^{\otimes \nu}(t))=H^i(Y,\Sym^2\mathcal U
\otimes\cQ^{\otimes \nu}(t))=0$ 
if $t \geq -4$ and $ i\geq 3$. 
\item 
$H^i(Y,\wedge^2\mathcal U\otimes\cQ^{\otimes \nu}(-5)) 
= 
H^i(Y,\Sym^2\mathcal U
\otimes\cQ^{\otimes \nu}(-5))=0$ 
for $i\geq 7$. 
\end{enumerate}
\end{lem}

\begin{proof}
We have the exact sequence 
$0 \to \cU \to \MO_Y^{\oplus 7} \to \cQ \to 0$ of the universal bundles. 
We then get 
the following exact sequence:
\[
0\to\mathcal U\otimes\cQ^{\otimes \nu}(t)
\to \cQ^{\otimes \nu}(t)^{\oplus 7}
\to  \cQ^{\otimes \nu+1}(t)
\to0. 
\]
Hence (1) and (2) follow from 
Lemma \ref{l genus 12 v6 polynomial vanishing}. 
By the exact sequence 
\[
0 \to \cU^{\otimes 2} \otimes \cQ^{\otimes \nu}(t) \to 
( \cU \otimes \cQ^{\otimes \nu}(t))^{\oplus 7} \to \cU \otimes \cQ^{\otimes \nu+1}(t) \to 0, 
\]
(1) (resp.\ (2)) implies (3) (resp.\ (4)) 
as $\wedge^2 \cU$ and $\Sym^2 \cU$ are direct summands 
of $\cU^{\otimes 2}$ in characteristic $p>2$. 
\end{proof}

\begin{lem}\label{l genus 12 qf mixed vanishing}
We use Notation \ref{n g=12}. 
Let $\nu$ be a non-negative integer and let
\[
\cV \in
\left\{
\mathcal U,\quad 
(\Sym^2\mathcal U)(1),\quad 
\wedge^2\mathcal U
\right\}.
\]
Then $H^i(
Y,
\cV\otimes\cQ^{\otimes \nu}(t)
)=0$ for every $i>0$ and every $t\geq 0$.
\end{lem}

\begin{proof}
The assertion follows from 
Proposition \ref{p g=12 good filt} and the exact sequence 
$0
\to
\wedge^2\cU
\to
\cU\otimes\cU
\to
\Sym^2\cU
\to
0.$ 
\end{proof}

\begin{lem}\label{l genus 12 qf ambient vanishing}
We use Notation \ref{n g=12}. 
Take integers $m\geq 3$ and $0 \leq j \leq 9$. 
Set
\[
\mathcal E_0:=(\wedge^2\mathcal Q)^{\oplus3}
\qquad
\text{and}
\qquad
K_j:=\wedge^j\mathcal E_0^\vee.
\]
%For every integer $m\geq3$ and every $0\leq j\leq9$, 
Then the following hold: 
\begin{enumerate}
\item 
$ 
H^{3+j}(
Y,
\bigl(\Sym^2\mathcal E_0^\vee\bigr)(m)
\otimes K_j
)=0$. 
\item 
$H^{2+j}(
Y,
(\mathcal E_0^\vee\otimes\Omega_Y^1)(m)
\otimes K_j
)=0$. 
\item $
H^{1+j}(
Y,
\Omega_Y^2(m)\otimes K_j
)=0$. 
\end{enumerate}

\end{lem}

\begin{proof}
We have $\mathcal E_0^\vee\simeq (\wedge^2\mathcal Q^\vee )^{\oplus3} \simeq \mathcal Q(-1)^{\oplus3}$ and 
\[
K_j 
= \wedge^j\mathcal E_0^\vee
\simeq  \wedge^j (\mathcal Q(-1)^{\oplus3})
\simeq
\bigoplus_{\substack{
a_1+a_2+a_3=j\\
0\leq a_1,a_2,a_3\leq3}}
\bigotimes_{\ell=1}^3
(
(\wedge^{a_\ell}\mathcal Q)(-a_\ell)
).
\]
As $(\wedge^3\mathcal Q)(-3) \simeq \MO_Y(-2)$, 
it holds that 
\[
\bigotimes_{\ell=1}^3
(
(\wedge^{a_\ell}\mathcal Q)(-a_\ell)
) = 
\mathcal Q^{\otimes u}
\otimes
(\wedge^2\mathcal Q)^{\otimes v} (-d)
\]
where $u,v,w$ are the numbers of the entries equal to $1,2,3$, 
respectively, and $d:=u+2v+2w\leq6$. We have 
$j=u+2v+3w=d+w \geq d$. 
Since 
$\wedge^2\mathcal Q$ is a direct summand of
$\mathcal Q^{\otimes2}$ by $p>2$, it is enough to prove (1)'-(3)' 
for all integers $\nu$ and  $d$ satisfying 
$\nu \geq 0$, $d \leq j$,  and $0 \leq d \leq 6$.  
\begin{enumerate}
\item[(1)'] 
$ 
H^{3+j}(
Y,
\bigl(\Sym^2\mathcal E_0^\vee\bigr)
\otimes \cQ^{\otimes \nu}(m-d)
)=0$. 
\item[(2)'] 
$H^{2+j}(
Y,
(\mathcal E_0^\vee\otimes\Omega_Y^1)
\otimes \cQ^{\otimes \nu}(m-d)
)=0$. 
\item[(3)'] $
H^{1+j}(
Y,
\Omega_Y^2 
\otimes \cQ^{\otimes \nu}(m-d)
)=0$. 
\end{enumerate}
By $m \geq 3$ and $d \leq 6$, the following hold:  
\begin{enumerate}
\item[(i)] $m-2-d\geq  3 -2 -6 \geq -5$. 
\item[(ii)] If $m-2-d=-5$, then $m =3$ and $j \geq d = 6$. 
\end{enumerate}

Let us show (1)'.  We have 
\[
\bigl(\Sym^2\mathcal E_0^\vee\bigr)(2)
\simeq
(\Sym^2\mathcal Q)^{\oplus3}
\oplus
(\mathcal Q^{\otimes2})^{\oplus3}.
\]
As $\Sym^2 \cQ$ is a direct summand of $\cQ^{\otimes 2}$, 
it suffices to prove 
\[
H^{3+j}(Y, \cQ^{\otimes 2} \otimes 
\mathcal Q^{\otimes \nu}
(m-2-d))=0. 
\]
If $m-2-d \geq -4$, then 
the assertion holds by Lemma \ref{l genus 12 v6 polynomial vanishing}(1). 
Hence we may assume $m-2-d =-5$ by (i). 
Then (ii) implies $j \geq 6$. 
In this case, we are done by  
Lemma \ref{l genus 12 v6 polynomial vanishing}(2). 
Thus (1)' holds. 

Let us show (2)'. By $\Omega_Y^1\simeq\mathcal U\otimes\mathcal Q^\vee$ and
$\mathcal Q^\vee\simeq(\wedge^2\mathcal Q)(-1)$, 
we obtain
\[
(\mathcal E_0^\vee\otimes\Omega_Y^1)(2)
\simeq
(
\mathcal U\otimes\mathcal Q\otimes
\wedge^2\mathcal Q
)^{\oplus3}.
\]
Hence 
%As $\Sym^2 \cQ$ is a direct summand of $\cQ^{\otimes 2}$, 
it suffices to prove 
\[
H^{2+j}(Y, 
\mathcal U \otimes 
\mathcal Q^{\otimes \nu +3} (m-2-d))=0. 
\]
If $m-2-d \geq -4$, then 
the assertion holds by Lemma \ref{l genus 12 v6 universal kernels}(1). 
Hence we may assume $m-2-d =-5$ by (i). 
Then (ii) implies $j \geq 6$. 
In this case, we are done by  
Lemma \ref{l genus 12 v6 universal kernels}(2). 
Thus (2)' holds.

Let us show (3)'. 
Recall that 
$\mathcal Q^\vee
\simeq
(\wedge^2\mathcal Q)(-1)$ and $\wedge^2\mathcal Q^\vee
\simeq
\mathcal Q(-1)$. 
Since $p>2$, we get 
\[
\Omega_Y^2(2)
\simeq
(
\wedge^2\mathcal U\otimes
\Sym^2\mathcal Q^\vee
)(2)
\oplus
(
\Sym^2\mathcal U\otimes
\wedge^2\mathcal Q^\vee
)(2) 
\simeq 
(
\wedge^2\mathcal U\otimes
\Sym^2 (\wedge^2 \cQ)
)
\oplus
(
\Sym^2\mathcal U\otimes
\mathcal Q(1)
)
\]
Since $\wedge^2 \cF$ and $\Sym^2 \cF$ are direct summands of $\cF \otimes \cF$, it is enough to prove 
\[
H^{1+j}(Y, \wedge^2 \cU \otimes \cQ^{\otimes \mu} (m-2-d)) 
= H^{1+j}(Y, \Sym^2\cU \otimes \cQ^{\otimes \mu} (m-1-d)) =0
\]
for every $\mu \geq 0$. 
If $j \in \{0, 1\}$, then 
this follows from Lemma \ref{l genus 12 qf mixed vanishing}, which is applicable by 
\[
m-2 -d \geq m -2 -j \geq 3-2-1=0. 
\]
Hence we may assume $j \geq 2$. 
By (i), we have $m-1-d \geq -4$, and hence 
Lemma \ref{l genus 12 v6 universal kernels}(3) implies 
$H^{1+j}(Y, \Sym^2\cU \otimes \cQ^{\otimes \mu} (m-1-d)) =0$. 
Similarly, if $m-2-d \geq -4$, then 
Lemma \ref{l genus 12 v6 universal kernels}(3) implies 
$H^{1+j}(Y, \wedge^2 \cU \otimes \cQ^{\otimes \mu} (m-2-d)) =0$. 
Thus it suffices to show 
$H^{1+j}(Y, \wedge^2 \cU \otimes \cQ^{\otimes \mu} (-5)) =0$ 
for the case when $m-2-d=-5$. 
This follows from 
Lemma \ref{l genus 12 v6 universal kernels}(4), 
because (ii) implies $1+j \geq 7$. 
Thus (3)' holds. 
\end{proof}

\begin{thm}\label{t genus 12 F-split}
Assume $p>2$. 
Let $X$ be a prime Fano threefold of genus $12$. 
%over an algebraicallyclosed field of characteristic $p>3$.
Then $X$ is $F$-split.
\end{thm}

\begin{proof}
By Proposition \ref{p quasi-F-split quadrics}, 
it is enough to show $H^1(X, \Omega^2_X(p))=0$. 
We use Notation \ref{n g=12}. 
By \cite[Theorem 1.2]{KTprime}, we can write
$X=Z(\sigma)\subset Y$ for some  section
$\sigma\in H^0(Y,\mathcal E_0)$, where 
$\mathcal E_0:=(\wedge^2\mathcal Q)^{\oplus3}$. 
It holds that $\mathcal O_X(-K_X)
\simeq
\mathcal O_Y(1)|_X$. 
%We also write $H$ for its restriction to $X$.
For $K_j:=\wedge^j\mathcal E_0^\vee$ ($0 \leq j \leq 9$), 
we have the following Koszul exact sequence associated with $\sigma$: 
\[
0
\to
K_9
\to
\cdots
\to
K_1
\to
K_0
\to
\mathcal O_X
\to
0.
\]
By taking the decomposition into the corresponding short exact sequences, we obtain $(\star)$ below. 
\begin{enumerate}
\item[($\star$)] 
Given an integer $n$ and a vector bundle $\cF$ on $Y$, 
if
$H^{n+j}(Y,\mathcal F\otimes K_j)=0$ for every $0\leq j\leq9$, 
then $H^n(X,\mathcal F|_X)=0$. 
\end{enumerate}
By applying $(\star)$ for 
\[
(n, \cF) = (3, \Sym^2 \cE_0^\vee(p)), \qquad 
(2, \cE_0^\vee \otimes \Omega_Y^1(p)), \qquad 
(1, \Omega_Y^2(p)), 
\]
it follows from Lemma \ref{l genus 12 qf ambient vanishing} that  %with $m=p$ gives
\[
H^3(
X,
\bigl(\Sym^2\mathcal E_0^\vee\bigr)(p)|_X
)= 
H^2(
X,
(\mathcal E_0^\vee\otimes\Omega_Y^1)(p)|_X
)=
H^1(
X,
\Omega_Y^2(p)|_X
)=0.
\]
Since $p>2$, 
the second exterior power of the conormal sequence $0
\to
\mathcal E_0^\vee|_X
\to
\Omega_Y^1|_X
\to
\Omega_X^1
\to
0$  gives the exact sequence 
\[
%\begin{aligned}
0
\to
\bigl(\Sym^2\mathcal E_0^\vee\bigr)(p)|_X
\to
(\mathcal E_0^\vee\otimes\Omega_Y^1)(p)|_X
\to
\Omega_Y^2(p)|_X
\to
\Omega_X^2(p)
\to
0.
%\end{aligned}
\]
By taking the decomposition into the corresponding short exact sequences, 
we obtain $H^1(X,\Omega_X^2(p))=0$, as required. 
\qedhere

\end{proof}

\begin{thm}\label{t prime Fsplit main}
Let $X$ be a prime Fano threefold of genus $g$. 
Then $X$ is $F$-split if $g$ and $p$ satisfy the following: 
{\small
\[
\begin{array}{c|cccccccccc}
g
& 2 & 3 & 4 & 5 & 6 & 7 & 8 & 9 & 10 & 12\\
\hline
\text{F-split}
& p>11 & p>7 & p>5 & p>3 & p>3
& p>2 & p>2 & p>2 & p>3 & p>2\\
% \text{GQFR}
% & p>5 & p>3 & p>2 & p>0 & p>0
% & p>0 & p>0 & p>0 & p>0 & p>0
\end{array}
\]
}
\end{thm}

\begin{proof}
If $|-K_X|$ is not very ample, then the assertion follows from  
\cite[Proposition 6.25]{KTLift2}. 
If $|-K_X|$ is very ample, then the assertion holds as follows: 
\begin{itemize}
\item $g=3$: Proposition \ref{p genus 3 F-split}. 
\item $g=4$: Proposition \ref{p genus 4 F-split}. 
\item $g=5$: Proposition \ref{p genus 5 Omega2}. 
\item $g=6$: Theorem \ref{t genus 6 F-split}. 
\item $g=7$: Theorem \ref{t genus 7 unified v1 F split}. 
\item $g=8$: \cite[Theorem 1.5]{KT26}. 
\item $g=9$: Theorem \ref{t genus 9 F-split v4}. 
\item $g=10$: Theorem \ref{t genus 10 F-split v2}. 
\item $g=12$: Theorem \ref{t genus 12 F-split}. 
\end{itemize}
\end{proof}

\begin{thm}\label{t Fsplit main}
Let $X$ be a Fano threefold. 
Assume $p>11$. 
Then $X$ is $F$-split. 
\end{thm}

\begin{proof}
If $X$ is a prime Fano threefold, 
then we are done by Theorem \ref{t prime Fsplit main}. 
Otherwise, the assertion follows from 
\cite[Theorem G]{KTLift2}. 
\end{proof}

\bibliographystyle{skalpha}
\bibliography{reference.bib}

\end{document}